\documentclass{amsart}

\usepackage{hyperref}
\usepackage{color,graphicx,pinlabel}
\usepackage{xypic,float,amsmath,amsfonts,amssymb}

\def\to{\mathchoice{\longrightarrow}{\rightarrow}{\rightarrow}{\rightarrow}}
\makeatletter
\newcommand{\shortxra}[2][]{\ext@arrow 0359\rightarrowfill@{#1}{#2}}
\newcommand{\shortxla}[2][]{\ext@arrow 3095\leftarrowfill@{#1}{#2}}
\def\longrightarrowfill@{\arrowfill@\relbar\relbar\longrightarrow}
\def\longleftarrowfill@{\arrowfill@\longleftarrow\relbar\relbar}
\newcommand{\longxra}[2][]{\ext@arrow 0359\longrightarrowfill@{#1}{#2}}
\newcommand{\longxla}[2][]{\ext@arrow 3059\longleftarrowfill@{#1}{#2}}
\renewcommand{\xrightarrow}[2][]{\mathchoice{\longxra[#1]{#2}}%
  {\shortxra[#1]{#2}}{\shortxra[#1]{#2}}{\shortxra[#1]{#2}}}
\renewcommand{\xleftarrow}[2][]{\mathchoice{\longxla[#1]{#2}}%
  {\shortxla[#1]{#2}}{\shortxla[#1]{#2}}{\shortxla[#1]{#2}}}
\makeatother

\makeatletter
\def\Nopagebreak{\@nobreaktrue\nopagebreak}
\makeatother

\theoremstyle{plain}
\newtheorem*{theorem*}{Theorem}
\newtheorem*{lemma*} {Lemma}
\newtheorem*{corollary*} {Corollary}
\newtheorem*{proposition*} {Proposition}
\newtheorem*{conjecture*}{Conjecture}

\newtheorem{theorem}{Theorem}[section]
\newtheorem{lemma}[theorem]{Lemma}
\newtheorem{corollary}[theorem]{Corollary}
\newtheorem{proposition}[theorem]{Proposition}

\theoremstyle{remark}

\newtheorem*{claim}{Claim}

\theoremstyle{definition}
\newtheorem{example}[theorem]{Example}
\newtheorem{remark}[theorem]{Remark}

\def\Z{\mathbb{Z}}
\def\R{\mathbb{R}}

\def\N{\mathbb{N}}
\def\SS{\mathcal{S}}
\def\inte{\operatorname{int}}
\def\id{\operatorname{id}}

\def\l{\lambda}
\def\ll{\langle}
\def\rr{\rangle}

\def\Hom{\operatorname{Hom}}
\def\rank{\operatorname{rank}}

\def\ba{\begin{array}}
\def\ea{\end{array}}

\def\L{\Lambda}

\def\g{\gamma}

\def\ti{\tilde}

\def\fr12{\frac{1}{2}}

\def\Aut{\operatorname{Aut}}
\def\ol{\overline}
\def\Im{\operatorname{Im}}
\def\Ker{\operatorname{Ker}}
\def\tr{\operatorname{tr}}

\def\BB{\mathcal{B}}
\def\CC{\mathcal{C}}
\def\CZ{\mathcal{C}_{\Z}}
\def\DD{\mathcal{D}}
\def\EE{\mathcal{E}}

\def\FF{\mathcal{F}}
\def\GG{\mathcal{G}}

\def\S{\Sigma}

\def\aut{\mbox{Aut}}

\def\zt{\Z[t^{\pm 1}]}

\def\ord{\operatorname{ord}}

\def\eps{\epsilon}

\def\zb{\Z[b^{\pm 1}]}
\def\zab{\Z[a^{\pm 1},b^{\pm 1}]}
\def\zbthena{\Z[b^{\pm 1}][a^{\pm 1}]}

\def\sgn{\S_{g,n}}
\def\CC{\mathcal{C}}
\def\MM{\mathcal{M}}
\def\HH{\mathcal{H}}
\def\II{\mathcal{I}}
\def\EE{\mathcal{E}}
\def\top{\mathrm{top}}
\def\smooth{\mathrm{smooth}}
\def\cgn{\CC_{g,n}}
\def\hgn{\mathcal{H}_{g,n}}
\def\ihgn{\mathcal{IH}_{g,n}}
\def\ign{\mathcal{I}_{g,n}}
\def\hgntop{\mathcal{H}_{g,n}^{\top}}

\def\hgnsmooth{\mathcal{H}_{g,n}^{\smooth}}
\def\hgone{\mathcal{H}_{g,1}}
\def\mgn{\mathcal{M}_{g,n}}
\def\mgzero{\mathcal{M}_{g,0}}
\def\zh{\Z[H]}
\def\mi{(M,i_+,i_-)}
\def\nj{(N,j_+,j_-)}
\def\S{\Sigma}

\def\AH{A}
\def\NH{N}
\def\ANH{AN}
\def\pmkgn{\mathcal{PM}^k_{g,n}}

\begin{document}

\title{The cobordism group of homology cylinders}


\author{Jae Choon Cha}

\address{Department of Mathematics and Pohang Mathematics Institute\\
  Pohang University of Science and Technology\\
  Pohang Gyungbuk 790--784\\
  Republic of Korea}

\email{jccha@postech.ac.kr}

\author{Stefan Friedl}

\address{Mathematics Institute\\
  University of Warwick\\
  Coventry CV4 7AL\\
  United Kingdom}

\email{s.k.friedl@warwick.ac.uk}

\author{Taehee Kim}

\address{Department of Mathematics\\
  Konkuk University\\
  Seoul 143--701\\
  Republic of Korea}

\email{tkim@konkuk.ac.kr}

\def\subjclassname{\textup{2000} Mathematics Subject Classification}
\expandafter\let\csname subjclassname@1991\endcsname=\subjclassname
\expandafter\let\csname subjclassname@2000\endcsname=\subjclassname


\begin{abstract}\leavevmode\kern.7em
Garoufalidis and Levine introduced the homology cobordism group of
homology cylinders over a surface.  This group can be regarded as an
enlargement of the mapping class group.  Using torsion invariants,
we show that the abelianization of this group is infinitely
generated provided that the first Betti number of the surface is
positive.  In particular, this shows that the group is not perfect.
This answers questions of Garoufalidis-Levine and Goda-Sakasai.
Furthermore we show that the abelianization of the group has
infinite rank for the case that the surface has more than one
boundary component.  These results hold for the homology cylinder
analogue of the Torelli group as well.
\end{abstract}

\maketitle

\section{Introduction}
\label{section:introduction}

Given $g\ge 0$ and $n\ge 0$, let $\sgn$ be a fixed oriented, connected
and compact surface of genus $g$ with $n$ boundary components.  We
denote by $\Hom^+(\sgn,\partial \sgn)$ the group of orientation
preserving diffeomorphisms of $\sgn$ which restrict to the identity on
the boundary. The \emph{mapping class group} $\mgn$ is defined to be
the set of isotopy classes of elements in $\mbox{Hom}^+(\sgn,\partial
\sgn)$, where the isotopies are understood to restrict to the identity
on the boundary as well. We refer to \cite[Section~2.1]{FM09} for
details. It is well known that the mapping class group is perfect
provided that $g\ge 3$ \cite{Po78} (e.g., see
\cite[Theorem~5.1]{FM09}) and that mapping class groups are finitely
presented \cite{BH71, Mc75} (e.g., see \cite[Section~5.2]{FM09}).

In this paper we intend to study an enlargement of the mapping class
group, namely the group of homology cobordism classes of homology
cylinders.  A \emph{homology cylinder} over $\sgn$ is roughly speaking
a cobordism between surfaces equipped with a diffeomorphism to $\sgn$
such that the cobordism is homologically a product.  Juxtaposing
homology cylinders gives rise to a monoid structure.  The notion of
homology cylinder was first introduced by Goussarov \cite{Go99} and
Habiro \cite{Ha00} (where it was referred to as a `homology
cobordism').

By considering smooth (respectively topological) homology cobordism
classes of homology cylinders we obtain a group $\hgnsmooth$
(respectively $\hgntop$).  These groups were introduced by
Garoufalidis and Levine \cite{GL05}, \cite{Le01}.  We refer to Section
\ref{section:homcyl} for the precise definitions of homology cylinder
and homology cobordism.  Henceforth, when a statement holds in both
smooth and topological cases, we will drop the decoration in the
notation and simply write $\hgn$ instead of $\hgntop$
and~$\hgnsmooth$.

It follows immediately from the definition that there exists a
canonical epimorphism $\hgnsmooth\to \hgntop$.  A consequence of work
of Fintushel-Stern \cite{FS90}, Furuta \cite{Fu90} and Freedman
\cite{Fr82} on smooth homology cobordism of homology $3$--spheres is
that this map is not an isomorphism.  In fact, using
their results we can see the following:

\begin{theorem}\label{thm:topsmooth}%
  Let $g,n\geq 0$. Then the kernel of the epimorphism
  $\HH_{g,n}^{\smooth}\to \HH_{g,n}^{\top}$ contains an abelian group
  of infinite rank. If $g=0$, then there exists in fact a homomorphism
  $\FF\colon\HH_{0,n}^{\smooth}\to A$ onto an abelian group of
  infinite rank such that  the restriction of $\FF$ to the kernel
  of the projection map $\HH_{0,n}^{\smooth}\to \HH_{0,n}^{\top}$ is
  also surjective.
\end{theorem}

An argument of Garoufalidis and Levine shows that the canonical map
$\mgn\to \hgn$ is injective.  (See also Proposition
\ref{prop:mgnembed}.)  It is a natural question which properties of
mapping class groups are carried over to~$\hgn$.  In particular in
\cite{GS09} Goda and Sakasai ask whether $\hgone^\smooth$ is a
perfect group and Garoufalidis and Levine
\cite[Section~5,~Question~9]{GL05} ask whether $\hgone^\smooth$ is
infinitely generated (see also \cite[Problem~11.4]{Mo06}).

The following theorem answers both questions:

\begin{theorem}\label{mainthm}
  If $b_1(\sgn)>0$, then there exists an epimorphism
  \[
  \hgn \to (\Z/2)^\infty
  \]
  which splits (i.e., there is a right inverse). In particular, the
  abelianization of $\hgn$ contains a direct summand isomorphic to
  $(\Z/2)^\infty$.
\end{theorem}

Note that Theorem~\ref{mainthm} also implies that $\hgn$ is not
finitely related, since for a finitely related group its
abelianization is also finitely related. Also refer to
Remark~\ref{remark:proof of mainthm} for a slightly more refined
statement.

In many cases we can actually strengthen the result:

\begin{theorem}\label{mainthm2}%
  If $n>1$, then there exists an epimorphism
  \[
  \hgn \to \Z^\infty.
  \]
  Furthermore, the abelianization of $\hgn$ contains a direct summand
  isomorphic to $(\Z/2)^\infty\bigoplus\Z^\infty$.
\end{theorem}

We remark that for the special case of $(g,n)=(0,2)$, this is a
consequence of Levine's work on knot concordance \cite{Le69a,Le69b}
since one can easily see that $\HH_{0,2}^{\top}$ maps onto Levine's
algebraic knot concordance group.
   The general cases
of Theorem~\ref{mainthm2} for $n>2$ and for $n=2$ and $g>0$ are new.

In order to prove Theorems \ref{mainthm} and \ref{mainthm2} we will
employ the torsion invariant of a homology cylinder, first introduced
by Sakasai (e.g., see \cite[Section~11.1.2]{Sa06},
\cite[Definition~6.5]{Sa08} and \cite[Definition~4.4]{GS08}).  In
Section \ref{section:torsion-invariant} we recall the definition of
the torsion of a homology cylinder and we study the behavior of
torsion under stacking and homology cobordism.  The result can be
summarized as a group homomorphism
\[
\tau\colon \hgn \to \frac{Q(H)^\times}{A(H)N(H)}.
\]
Here $H=H_1(\Sigma_{g,n})$ and $Q(H)^\times$ is the multiplicative
group of nonzero elements in the quotient field of the group ring
$\Z[H]$.  Loosely speaking, $A(H)$ reflects the action of surface
automorphisms on $H$, and $N(H)$ is the subgroup of ``norms''
in $Q(H)^\times$.  (For details, see
Section~\ref{section:torsion-invariant}.)

An interesting point is that torsion invariants of homology cylinders
may be \emph{asymmetric}, in contrast to the symmetry of
the Alexander polynomial of knots.  Indeed, in
Section~\ref{section:homomorphisms}, we extract infinitely many
$(\Z/2)$--valued and $\Z$--valued homomorphisms of
$Q(H)^\times/A(H)N(H)$ from symmetric and asymmetric irreducible
factors, respectively. Theorems \ref{mainthm} and \ref{mainthm2} now
follow from the explicit construction of examples in Section
\ref{section:construction-computation}.

We remark that Theorem~\ref{mainthm2} covers all the possibilities of
the asymmetric case, since it can be seen that for either $n<2$ or
$(g,n)=(0,2)$ the torsion of a homology cylinder over $\Sigma_{g,n}$
is always symmetric in an appropriate sense (see
Section~\ref{subsection:symmetry-asymmetry} for details.)

We remark that our main results hold even modulo the mapping class
group---the essential reason is that our torsion invariant is trivial
for homology cylinders associated to mapping cylinders.  More
precisely, denoting by $\langle\mgn\rangle$ the normal subgroup of
$\hgn$ generated by $\mgn$, the torsion homomorphism $\tau$ actually
factors as
\[
\tau\colon \hgn \to \frac{\hgn}{\langle\mgn\rangle} \to
\frac{Q(H)^\times}{A(H)N(H)}
\]
and Theorems~\ref{mainthm} and~\ref{mainthm2} hold for $\hgn/\langle
\mgn\rangle$ as well as~$\hgn$.

In Section \ref{section:pretzel}, we study examples of homology
cylinders which naturally arise from Seifert surfaces of pretzel
links.  We compute the torsion invariant and prove that these homology
cylinders span a $\Z^\infty$ summand in the abelianization~of
$\HH_{0,3}$.

Finally in Section~\ref{section:torelli}, we consider the ``Torelli
subgroup'' $\ihgn$ of $\hgn$ which is the homology cylinder analogue
of the Torelli subgroup of mapping cylinders.  We prove that the
conclusions of Theorems \ref{mainthm} and \ref{mainthm2} hold for
the Torelli subgroup~$\ihgn$, but for a larger group of surfaces.
(See Theorem~\ref{thm:cyltorelli} for details.) This extends work of
Morita's \cite[Corollary~5.2]{Mo08} to a larger class of surfaces.

In this paper, manifolds are assumed to be compact, connected, and
oriented. All homology groups are with respect to integral
coefficients unless it says explicitly otherwise.

\subsection*{Acknowledgment}

This project was initiated while the three authors visited the
Mathematics Department of Indiana University at Bloomington in July
2009.  We wish to express our gratitude for the hospitality and we
are especially grateful to Chuck Livingston and Kent Orr.  The
authors would like to thank Irida Altman, Jake Rasmussen, Dan Silver
and Susan Williams for helpful conversations. We are very grateful
to Takuya Sakasai for suggesting a stronger version of Theorem
\ref{thm:topsmooth} to us and for informing us of Morita's work in
\cite{Mo08}. The authors thank an anonymous referee for helpful
comments, in particular the comment about
Proposition~\ref{prop:mgnembed}. The first author was supported by
the National Research Foundation of Korea(NRF) grant funded by the
Korea government(MEST) (No. 2007-0054656 and 2009-0094069). The last
author was supported by the National Research Foundation of
Korea(NRF) grant funded by the Korea government(MEST) (No.
2009-0068877 and 2009-0086441).

\section{Homology cylinders and their cobordism}
\label{section:homcyl}

In this section we recall basic definitions and preliminaries on
homology cylinders, the notion of which goes back to Goussarov \cite{Go99},
Habiro \cite{Ha00}, and Garoufalidis and Levine \cite{GL05},
\cite{Le01}.

\subsection{Cobordism classes of homology cylinders}

Given $g\geq 0$ and $n\geq 0$ we fix, once and for all, a surface
$\Sigma_{g,n}$ of genus $g$ with $n$ boundary components.  When $g$
and $n$ are obvious from the context, we often denote $\Sigma_{g,n}$
by~$\Sigma$.  A \emph{homology cylinder $(M,i_+,i_-)$ over $\Sigma$}
is defined to be a $3$--manifold $M$ together with injections $i_+$,
$i_- \colon \Sigma \to \partial M$ satisfying the following:

\begin{enumerate}
\item $i_+$ is orientation preserving and $i_-$ is orientation
  reversing.
\item $\partial M=i_+(\Sigma)\cup i_-(\Sigma)$ and $i_+(\Sigma)\cap
  i_-(\Sigma)=i_+(\partial \Sigma)=i_-(\partial \Sigma)$.
\item $i_+|_{\partial \Sigma}=i_-|_{\partial \Sigma}$.
\item $i_+,i_-\colon H_*(\Sigma)\to H_*(M)$ are isomorphisms.
\end{enumerate}

\begin{example}
  \label{example:homology-cylinder}
\leavevmode\Nopagebreak
\begin{enumerate}
\item Let $\varphi \in \mgn$.
  Then $\varphi$ gives rise to a homology cylinder
  \[
  M(\varphi)=(\sgn\times [0,1]/\mathord{\sim},\,
  i_+=\mathord{\id}\times 0 ,\, i_-=\varphi\times 1).
  \]
  where $\sim$ is given by $(x,s)\sim(x,t)$ for $x\in \partial \sgn$
  and $s,t\in [0,1]$. If
  $\varphi$ is the identity, then we will refer to the resulting
  homology cylinder as the \emph{product homology cylinder}.
\item Let $K$ be a knot of genus $g$ such that the Alexander
  polynomial $\Delta_K$ is monic and of degree~$2g$.  Let $\Sigma$ be
  a minimal genus Seifert surface in the exterior $X$ of~$K$. Then $X$
  cut along $\Sigma$ is a homology cylinder over $\S_{g,1}$ in a
  natural way (e.g., see \cite[Proposition~3.1]{Ni07} or \cite{GS08}
  for details).
\end{enumerate}
\end{example}

Two homology cylinders $(M,i_+,i_-)$ and $(N,j_+,j_-)$ over
$\Sigma=\Sigma_{g,n}$ are called \emph{isomorphic} if there exists an
orientation-preserving diffeomorphism $f\colon M\to N$ satisfying
$j_\pm=f\circ i_\pm$.  We denote by $\CC_{g,n}$ the set of all
isomorphism classes of homology cylinders over $\sgn$.  A product
operation on $\cgn$ is given by stacking:
\[
(M,i_+,i_-)\cdot (N,j_+,j_-):=(M\cup_{i_-\circ (j_+)^{-1}} N, i_+, j_-).
\]
This turns $\cgn$ into a monoid.  The unit element is given by the
product homology cylinder.

Two homology cylinders $(M,i_+,i_-)$ and $(N,j_+,j_-)$ over $\Sigma$
are called \emph{smoothly homology cobordant} if there exists a
compact oriented \emph{smooth} $4$--manifold $W$ such that
\[
\partial W=M\cup (-N)\, /\, i_+(x)=j_+(x),\, i_-(x)=j_-(x) \quad (x\in \Sigma),
\]
and such that the inclusion induced maps $H_*(M)\to H_*(W)$ and
$H_*(N)\to H_*(W)$ are isomorphisms.  We denote by $\hgnsmooth$ the
set of smooth homology cobordism classes of elements in~$\cgn$. The monoid
structure on $\cgn$ descends to a group structure on $\hgnsmooth$,
where the inverse of a mapping cylinder $(M,i_+,i_-)$ is given by
$(-M,i_-,i_+)$.  (We refer to \cite[p.~246]{Le01} for details).

If there is a topological $4$--manifold $W$ satisfying the above
conditions, then we say that $(M,i_+,i_-)$ and $(N,j_+,j_-)$ are
\emph{topologically homology cobordant}.  We denote the resulting
group of homology cobordism classes by $\hgntop$.  Note that there
exists a canonical surjection $\hgnsmooth\to \hgntop$. We will see in
the following sections that this map is in general not an isomorphism.
\begin{remark}
  In the original papers of Garoufalidis and Levine \cite{GL05, Le01}
  and in \cite{GS09} the authors focus on the smooth case and denote
  the group $\hgn^\smooth$ by~$\hgn$.
\end{remark}

\subsection{Examples}\label{section:examples}
  In this
section we will discuss three types of examples:
\begin{enumerate}
\item surface automorphisms,
\item homology cobordism classes of integral homology spheres, and
\item concordance classes of (framed) knots in integral homology spheres.
\end{enumerate}

\subsubsection*{Surface automorphisms and homology cylinders}
First recall that $\varphi\in \mgn$ gives rise to a homology
cylinder~$M(\varphi)$.  Note that if $\varphi,\psi\in \mgn$, then
$M(\varphi)\cdot M(\psi)$ is isomorphic to $M(\varphi\circ \psi)$.
In particular the map $\mgn\to \cgn$ descends to a morphism of
monoids. Proposition~\ref{prop:mgnembed} below says that we can view
the mapping class group $\mgn$ as a subgroup of $\hgn$. To prove the
proposition, we need the following folklore theorem:

\begin{theorem}\label{thm:folklore}
Let $\S$ be a surface with one boundary component, possibly with
punctures. Let $*\in \partial \S$ be a base point. Suppose $h\colon \S\to \S$ is a homeomorphism
with the following properties:
    \begin{enumerate}
    \item $h$ is the identity on $\partial \S$,
    \item $h_*:\pi_1(\S,*) \to \pi_1(\S,*)$ is the identity, and
    \item $h$ fixes each puncture.
    \end{enumerate}
Then $h$ is isotopic to the identity where the isotopy restricts to
the identity on $\partial \S$.
\end{theorem}

\begin{proof}
The theorem is well-known, and we only give an outline of the proof. We
choose disjoint circles $\g_i$
in $\S$ based at $*$ such that $\S$ cut along the $\g_i$ is a
punctured disk. Since $h_*$ is the identity, we may assume (after applying an isotopy whose restriction to $\partial \S$
is the identity) that $h$
fixes each $\g_i$. Now the map on the punctured disk induced from $h$
is isotopic to the identity with $\partial \S$ fixed pointwise. This
can be shown, for example, using the fact that the map from the pure
braid group to $\pi_1(\S)$ is injective.
\end{proof}

\begin{proposition} \label{prop:mgnembed}
The  map $\mgn\to \cgn\to \hgn$ is injective.
\end{proposition}

\begin{proof}
The proposition was proved by Garoufalidis and Levine in the case
that $n=1$ and in the smooth category.  (See
\cite[Section~2.4]{GL05} and \cite[Section~2.1]{Le01}.) Using their
arguments partially, we will show that the proposition holds for any
$n\ge 0$ (and in the topological category as well).

First note that if $(g,n)=(0,0)\mbox{ or } (0,1)$, i.e., if $\sgn$
is a sphere or a disk, then it is well-known that any orientation
preserving diffeomorphism is isotopic to the identity. If
$(g,n)=(0,2)$, i.e., if $\sgn$ is an annulus, then it is known that
$\mgn\cong \Z$ and it injects into $\hgn$ (for example, this can be
shown using the arguments in the subsection below entitled {\it
Concordance of (framed) knots in homology $3$--spheres}). Therefore
we henceforth assume that $(g,n)\ne (0,0), (0,1), (0,2)$.

Suppose $\varphi\colon \sgn\to \sgn$ is an orientation-preserving
diffeomorphism such that
    \begin{enumerate}
    \item $\varphi$ restricts to the identity on $\partial \sgn$,
    and
    \item $[\varphi]\in \Ker\{\mgn\to \hgn\}$.
    \end{enumerate}
Now fix $k\in \{1,2,\ldots, n\}$ and fix a base point $*$ lying in
the $k$-th component, say $\partial_k$, of $\partial \sgn$. We write
$\pi:= \pi_1(\sgn,\,\, *)$. We denote by $\pi_l$ the lower central
series of $\pi$ defined inductively by $\pi_1:=\pi$ and
$\pi_l:=[\pi,\pi_{l-1}]$, $l>1$. The argument of Garoufalidis and
Levine (which builds on Stallings' theorem \cite{St65}) shows that
if $M(\varphi)$ is homology cobordant to the product homology
cylinder, then $\varphi_*\colon\pi/\pi_l\to \pi/\pi_l$ is the
identity map for any~$l$. Since $\bigcap_l \pi_l$ is trivial, this
implies that $\varphi_*\colon \pi \to \pi$ is the identity.

We denote by $\pmkgn$ the set of equivalence classes of
orientation-preserving diffeomorphisms $h\colon \sgn \to \sgn$ such
that
    \begin{enumerate}
    \item $h$ is the identity on $\partial_k$, and
    \item $h$ fixes each boundary component setwise,
    \end{enumerate}
where we say that two such maps are equivalent if they are related
by an isotopy which fixes  $\partial_k$ pointwise. Then
$[\varphi]=[\mbox{identity}]$ in $\pmkgn$ by
Theorem~\ref{thm:folklore}.

Now consider the map $\mgn\to \pmkgn$. It is known that the $n$ Dehn
twists along boundary components of $\sgn$ generate a central
subgroup isomorphic to $\Z^n$ in $\mgn$, and $\Ker\{\mgn\to
\pmkgn\}$ is the subgroup generated by the $n-1$ Dehn twists along all, but the $k$-th boundary component. (For example, see \cite[Section~4.2]{FM09}.)
Therefore, $[\varphi]\in \bigcap_k \Ker\{\mgn\to \pmkgn\} =
\{[\mbox{identity}]\}$ in $\mgn$.
\end{proof}

Note that this shows that $\hgn$ is non-abelian provided that $g>0$ or
$n>2$. It is straightforward to see that in the remaining cases $\hgn$
is abelian.

In the following we will see that the cobordism groups of homology
cylinders over the surfaces $\S_{0,n}$ with $n=0,1,2$ have been
studied under different names for many years.

\subsubsection*{Homology cobordism of integral homology $3$--spheres}
We first consider the case $n=0,1$.  Recall that oriented integral
homology $3$--spheres form a monoid under the connected sum operation.
Two oriented integral homology $3$--spheres $Y_1$ and $Y_2$ are called
smoothly (respectively topologically) cobordant if there exists a
smooth (respectively topological) $4$--manifold cobounding $Y_1$
and~$-Y_2$.  We denote by $\Theta_{3}^\smooth$ (respectively
$\Theta_3^{\top}$) the group of smooth (respectively topological)
cobordism classes of integral homology $3$--spheres.

For $n=0,1$ the group $\HH_{0,n}^\smooth$ (respectively
$\HH_{0,n}^{\top}$) is naturally isomorphic to the group
$\Theta_3^\smooth$ (respectively $\Theta_3^{\top}$) (e.g., see\
\cite[p.~59]{Sa06}).  Furuta \cite{Fu90} and Fintushel-Stern
\cite{FS90} showed that $\Theta_{3}^\smooth$ has infinite rank.  (See
also \cite[Section~7.2]{Sav02}.)  On the other hand it follows from
the work of Freedman \cite{Fr82} that
$\HH_{0,n}^{\top}=\Theta_{3}^{\top}$ is the trivial group.  (See also
\cite[Corollary~9.3C]{FQ90}.)  This shows in particular that the
homomorphism $\HH_{0,n}^\smooth\to \HH_{0,n}^\top$ is not an
isomorphism for $n=0,1$.

\subsubsection*{Concordance of (framed) knots in homology $3$--spheres}
We now turn to the case $n=2$, i.e., homology cylinders over the
surface $\Sigma_{0,2}$ which we henceforth identify with the annulus
$S^1\times [0,1]$.  Let $K\subset Y$ be an oriented knot in an
integral homology $3$--sphere. Let $M$ be the exterior of~$K$.  It is
not difficult to see there are pairs $(i_+, i_-)$ of maps
$\S_{0,2}\to \partial M$ satisfying (1)--(4) in the definition of a
homology cylinder and satisfying the condition that $i_+(S^1\times 0)$
is a meridian of~$K$.  Furthermore, the isotopy types of such $(i_+,
i_-)$ are in 1--1 correspondence with framings on~$K$.  Indeed, the
linking number of $K$ and the closed curve $i_-(*\times[0,1])\cup
i_+(*\times[0,1])$ gives rise to a canonical 1--1 correspondence
between the set of framings and~$\Z$.  Conversely, a homology cylinder
$(M,i_+,i_-)$ over $\Sigma_{0,2}$ determines an oriented knot endowed
with a framing in the integral homology sphere, which is given by
attaching a $2$--handle along $i_+(S^1\times 0)$ and then attaching a
$3$--handle.

We say that $K_1\subset Y_1$ and $K_2\subset Y_2$ are
\emph{smoothly concordant} if there exists a smooth cobordism $X$ of
$Y_1$ and $Y_2$ such that $(X,Y_1,Y_2)$ is an integral homology
$(S^3\times[0,1], S^3\times 0, S^3\times 1)$ and $X$ contains a
smoothly embedded annulus $C$ cobounding $K_1$ and~$K_2$.  The set of
smooth concordance classes of knots in integral homology $3$--spheres
form a group $\CZ^\smooth$ under connected sum.

We will now see that we can also think similarly of the concordance
group of framed knots in integral homology $3$--spheres.  Note that
a concordance $(X,C)$ as above determines a 1--1 correspondence
between framings on $K_1$ and~$K_2$.  We say two framed knots in
integral homology spheres are \emph{smoothly concordant} if there is
a concordance $(X,C)$ via which the given framings correspond to
each other.  Smooth concordance class of framed knots form a group
as well, and it is easily seen that this framed analogue of
$\CZ^{\smooth}$ is isomorphic to $\Z\oplus\CZ^\smooth$.  Similarly,
if we allow topologically locally flat annuli in topological
cobordisms we obtain a group $\CZ^{\top}$ and its framed analogue
$\Z\oplus\CZ^{\top}$.

As usual we adopt the convention that the group $\CZ$ can mean
either $\CZ^{\smooth}$ or $\CZ^{\top}$.  It follows easily from the
definitions that we have an isomorphism $\Z\oplus \CZ\to \HH_{0,2}$.
It follows from the work of Levine \cite{Le69a,Le69b} that $\CZ$
maps onto the algebraic knot concordance group which is isomorphic
to $(\Z/2)^{\infty}\oplus (\Z/4)^{\infty}\oplus \Z^\infty$, and
furthermore the epimorphism from $\CZ$ to $(\Z/2)^\infty\oplus
\Z^{\infty}$ splits. This discussion in particular proves the
following special case of Theorem \ref{mainthm2}:

\begin{theorem}\label{mainthm2-special}
  There exists a split surjection of $\HH_{0,2}$ onto $(\Z/2)^{\infty}
  \oplus \Z^\infty$.
\end{theorem}

Note that the subgroup $\MM_{0,2}$ in $\HH_{0,2}\cong \Z\oplus \CZ$
is exactly the $\Z$ factor, so that $\HH_{0,2}/\MM_{0,2} \cong \CZ$.
In particular, Theorem~\ref{mainthm2-special} holds for
$\HH_{0,2}/\MM_{0,2}$ as well as~$\HH_{0,2}$.

\subsection{Proof of Theorem \ref{thm:topsmooth}}

Before we turn to the proof of Theorem \ref{thm:topsmooth} we
introduce a gluing operation on homology cylinders.  Let $M$ be a
homology cylinder over a surface $\S=\S_{g,n}$ and let $M'$ be a
homology cylinder over a surface $\S'=\S'_{g',n'}$.  Assume that
$n,n'>0$ and fix a boundary component $c$ of $\S$ and fix a boundary
component $c'$ of $\S'$.  We can now glue $\S$ and $\S'$ along $c$ and
$c'$ using an orientation reversing homeomorphism.  Similarly we can
glue $M$ and $M'$ along neighborhoods of $c\subset \partial M$ and
$c'\subset \partial M'$.  This gives a homology cylinder over
$\S\cup_{c=c'} \S'$, which we denote by $M\cup_{c,c'}
M'$.  We refer to
$M\cup_{c,c'} M'$ as the \emph{union of $M$ and $M'$ along $c$ and
  $c'$}.

Now let $M'$ be the product homology cylinder over $\S'$.  The
association $M\mapsto M\cup_{c,c'}M'$ gives rise to a monoid
homomorphism $\cgn \to \CC_{g+g',n+n'-2}$.  We refer to it as an
\emph{expansion by $\S'$ along $c$}.  Note that the expansion map
descends to a group homomorphism $\hgn \to \HH_{g+g',n+n'-2}$.

We are now in a position to prove Theorem \ref{thm:topsmooth}.

\begin{proof}[Proof of Theorem \ref{thm:topsmooth}]
We adopt the convention that $\Theta_3$ stands either for
$\Theta_3^{\smooth}$ or $\Theta_3^{\top}$. Recall that in Section
\ref{section:examples} we saw that we can identify $\HH_{0,n},
n=0,1,$ with $\Theta_3$. Also recall that by   Section
\ref{section:examples} the group  $\Theta_{3}^\smooth$ has infinite
rank and that $\Theta_{3}^{\top}$ is the trivial group.

As we saw above,  an expansion by a surface of genus $g$  with $n+1$
punctures gives rise to a homomorphism $\EE\colon \HH_{0,1}\to
\HH_{g,n}$. We also consider the composition $\mathcal{F}\colon
\HH_{g,n}\to \HH_{g,0}$ of $n$ expansions by a disk along all
boundary components of $\S_{g,n}$. Loosely speaking, $\mathcal{F}$
is the homomorphism given by filling in the $n$ holes.

\begin{claim}
There exists a map $\GG:\HH_{g,0}\to \HH_{0,0}$ such that the
composition
\[ \Theta_3=\HH_{0,1} \xrightarrow{\EE} \HH_{g,n}\xrightarrow{\FF}
\HH_{g,0}\xrightarrow{\GG}  \HH_{0,0}=\Theta_3\]
is the identity.
\end{claim}

Let $\Lambda=\S_{g,0}$. Let $Y$ be a fixed handlebody of genus $g$.
Given $i=1,2$ we write $Y_i=Y$. Since $S^3$ has a Heegaard
decomposition of genus $g$ there  exist diffeomorphisms
$\varsigma_i:\Lambda\to \partial Y_i$ such that $Y_1
\cup_{\varsigma_1\circ \varsigma_2^{-1}} Y_2=S^3$.

We write $H=H_1(\L)$. We will see in Section \ref{section:actionh1}
that the action of a homology cylinder $M$ over $\L$ on $H_1(\L)$
gives rise to a homomorphism $\varphi:\mgzero\to \aut(H)$. We will
write $\aut^*(H)=\varphi(\mgzero)$. We now pick a splitting map
$\psi:\aut^*(H)\to \mgzero$, i.e. a map such that $\varphi\circ
\psi$ is the identity on $\aut^*(H)$. We pick $\psi$ such that
$\psi(\id)=\id$.

 Note that we can not arrange
that $\psi$ is a homomorphism. Now consider the following map
\[ \ba{rcl} \GG:\CC_{g,0} &\mapsto &\{ \Z\mbox{-homology spheres}\} \\
(M,i_+,i_-)&\mapsto & Y_1 \cup_{\varsigma_1 \circ i_+^{-1}} M
\cup_{i_- \circ \psi(\varphi(M))^{-1}\circ \varsigma_2^{-1}} Y_2.\ea
\] Note that this map is indeed well-defined, i.e. the right hand
side is an integral homology 3--sphere. Also note that this map
descends to a map $\HH_{g,0}\to \Theta_3$. It is easy to verify that
\[
\Theta_3=\HH_{0,1} \xrightarrow{\EE}  \HH_{g,n}\xrightarrow{\FF}
\HH_{g,0}\xrightarrow{\GG}  \HH_{0,0}=\Theta_3
\]
is indeed the identity map. This concludes the proof of the claim.

Before we continue we point out that the map $\GG$ is in general
\emph{not} a monoid morphism.
  We now obtain the following commutative
  diagram:
  \[ \xymatrix{ \Theta_3^{\smooth} \ar[r]^{=}\ar[d]
  &\HH_{0,1}^\smooth\ar[r]^{\EE}\ar[d]&
    \HH_{g,n}^\smooth\ar[d]\ar[r]^-{\FF}
    &\HH_{g,0}^\smooth\ar[d]\ar[r]^-{\GG}
    &\HH_{0,0}^\smooth\ar[d]\ar[r]^{=}& \Theta^{\smooth}_3\ar[d]\\
\Theta_3^{\top} \ar[r]^{=} &    \HH_{0,1}^{\top}\ar[r]^{\EE}&
\HH_{g,n}^{\top}\ar[r]^{\FF}& \HH_{g,0}^{\top}\ar[r]^{\GG}&
    \HH_{0,0}^{\top}\ar[r]^{=}& \Theta^{\top}_3.}
  \]
Since $\GG \circ \FF\circ \EE=\id$ we deduce that
$\Theta_3^{\smooth}\xrightarrow{\EE} \hgn^{\smooth}$ is injective.
Furthermore it follows from  $\Theta_3^{\top}=0$ and from the above
diagram that $\EE(\Theta_3^{\smooth})$ is contained in  the kernel
of the projection map $\HH_{g,n}^\smooth\to \hgn^{\top}$. Since
$\Theta_3^{\smooth}$ has infinite rank this concludes the proof of
the first part of Theorem \ref{thm:topsmooth}.

If $g=0$, then we can take $\GG$ to be the identity map. In
particular all maps are homomorphisms, and the homomorphism
$\FF:\HH_{0,n}^{\smooth}\to \Theta_3^{smooth}=:A$ has the desired
properties.
\end{proof}


\section{Torsion invariants of homology cylinders}
\label{section:torsion-invariant}

\subsection{The  torsion invariant}

Let $(M,N)$ be a pair of manifolds. Let $\varphi\colon \pi_1(M)\to H$
be a homomorphism to a free abelian group.  We denote the quotient
field of $\Z[H]$ by $Q(H)$. Denote by $p\colon \ti{M}\to M$ the
universal covering of $M$ and write $\ti{N}:=p^{-1}(N)$. Then we can
consider the chain complex
\[
C_*(M,N;Q(H))=C_*(\ti{M},\ti{N};\Z)\otimes_{\Z[\pi_1(M)]}Q(H).
\]
Let $\BB_*$ be a basis for the graded vector space $H_*(M,N;Q(H))$,
then we obtain the corresponding torsion
\[
\tau_{\BB_*}(M,N;Q(H))\in Q(H)^\times := Q(H)-\{0\}.
\]
This torsion invariant is well-defined up to multiplication by an
element of the form $\pm h$ ($h\in H)$. If $(M,N)$ is $Q(H)$--acyclic,
then $\BB_*$ is the trivial basis and we just write $\tau(M,N;Q(H))$.
We will not recall the definition of torsion but refer instead to the
many excellent expositions, e.g., \cite{Mi66}, \cite{Tu01}, \cite{Tu02}
and \cite{Nic03}.

We will several times make use of the following well-known lemma
(e.g., see \cite[Proposition~2.3]{KLW01} for a proof).

\begin{lemma}\label{lem:alsozero}
  If $H_*(M,N)=0$ and if $\pi_1(M)\to H$ is a homomorphism to a free
  abelian group, then $H_*(M,N;Q(H))=0$.
\end{lemma}

We also adopt the following notation: given $p,q\in Q(H)$ we write
$p\doteq q$ if $p=\eps h \cdot q$ for some $\eps\in \{-1,1\}$ and
$h\in H$. Put differently, $p\doteq q$ if and only if $p$ and $q$
agree up to multiplication by a unit in $\zh$.

\subsection{Torsion of homology cylinders}

Let $(M,i_+,i_-)$ be a homology cylinder over~$\Sigma$.
Write $H=H_1(\Sigma)$. We normally think of $H$ as a multiplicative group.
 Denote
$\Sigma_\pm = i_\pm(\Sigma) \subset M$.  Consider
\[
\varphi\colon \pi_1(M)\to H_1(M)\xleftarrow{\cong}
H_1(\S_+)\xleftarrow{i_+} H=H_1(\Sigma).
\]
Since $H_*(M,\S_+)=0$ it follows from Lemma \ref{lem:alsozero} that
$H_*(M,\S_+;Q(H))=0$.  Therefore we can define
\[
\tau(M):=\tau(M,\S_+;Q(H)) \in Q(H)^\times.
\]
This is referred to as the \emph{torsion} of the homology cylinder
$M=(M,i_+,i_-)$.

We will now show that the torsion can be defined in terms of homology.
First note that $\Z[H]$ is a unique factorization domain, so that for any
finitely generated $\Z[H]$-module $M$, the \emph{order} of $M$ is
defined as an element of $\Z[H]$.  We denote it by $\ord M$.  (For a
precise definition, e.g., see \cite[Section 4.2]{Tu01}.)  Note that
$\ord M$ is well-defined up to multiplication by a unit in~$\zh$.

\begin{lemma}\label{lemma:homology-and-torsion}
  For any homology cylinder $M$ we have $\tau(M) \doteq \ord
  H_1(M,\Sigma_+;\Z[H])$.
\end{lemma}
\begin{proof}
  It is well known that $\tau(M) = \prod_i (\ord
  H_i(M,\Sigma_+;\Z[H]))^{(-1)^{i+1}}$ (e.g., see \cite[Section
  4.2]{Tu01}). Since $H_1(\S_+)\to H_1(M)$ is an isomorphism it
  follows immediately that $H_0(M,\S_+;\Z[H])=0$. One can check that
  $H_i(M,\Sigma_+;\Z[H])=0$ for all $i>1$, by using Poincar\'e
  duality, universal coefficient spectral sequence, and the fact that
  $H_1(M,\Sigma_-;\Z[H])$ is torsion.
\end{proof}

\begin{remark} \label{remark:torsion}
  \leavevmode\Nopagebreak
  \begin{enumerate}
  \item The torsion of a homology cylinder was first studied by
    Sakasai \cite{Sa06} and is closely related to the torsion
    invariants of sutured manifold introduced independently by
    Benedetti and Petronio \cite{BP01} and \cite{FJR09}.
  \item Note that string links give rise to homology cylinders in a
    natural way. In this context the torsion invariant was first
    studied by Kirk, Livingston and Wang \cite[Definition~6.8]{KLW01}.
  \item By Lemma~\ref{lemma:homology-and-torsion} (see also
    \cite[Lemma~3.5]{FJR09}) the torsion $\tau(M)$ is in fact an
    element in $\Z[H]$, furthermore, if $\eps\colon \Z[H]\to \Z$
    denotes the augmentation defined by $\eps(h)=1$ for $h\in H$, then
    $\eps(\tau(M))=|H_1(M,\S_+)|=1$ (e.g., see \cite{Tu86} and
    \cite[Proposition~5]{FJR09}).
  \item By Lemma~\ref{lemma:homology-and-torsion}, if $M=\mi$ is the
    homology cylinder over an annulus corresponding to a knot
    $K\subset S^3$ as in Section~\ref{section:examples}, then
    $\tau(M)=\Delta_K$, the Alexander polynomial of $K$.
  \item If $\varphi\in \mgn$, then $H_*(M(\varphi),\S_+;\Z[H])=0$ and
    therefore $\tau(M(\varphi))\doteq 1\in \Z[H]$.
  \end{enumerate}
\end{remark}

\subsection{Action of homology cylinders on $H_1(\Sigma)$}
\label{section:actionh1}

Before we can state the behavior of torsion under the product
operation we have to
study the action of a homology cylinder on the first homology of the
surface.  Our notation is as follows.  Throughout this paper we
write $H=H_1(\Sigma)$.  Given a homology cylinder $(M,i_+,i_-)$ over
$\Sigma$ we denote the automorphism
\[
(i_+)_*^{-1}(i_-)_*\colon H=H_1(\S)\xrightarrow[(i_-)_*]{\cong}
H_1(M) \xrightarrow[(i_+)_*^{-1}]{\cong} H
\]
by $\varphi(M)=\varphi(M,i_+,i_-)$.  Let $H_\partial$ be the image of
$H_1(\partial \Sigma)$ in $H$, and let
\[
\Aut^*(H) = \{\varphi\in \Aut(H) \mid \text{$\varphi$ fixes
  $H_\partial$ and preserves the intersection form of $\Sigma$}\}.
\]

We recall the following result:
\def\tempstr{\cite[Proposition 2.3 and Remark 2.4]{GS08}}
\begin{proposition}
  [Goda-Sakasai \tempstr]
  \label{prop:inauthdot}
  Let $M$ be a homology cylinder over $\Sigma$.  Then $\varphi(M) \in
  \Aut^*(H)$.
\end{proposition}

Let $\widehat H=H/H_\partial = H_1(\widehat \Sigma)$ where
$\widehat\Sigma$ denotes $\Sigma$ with capped boundary circles.  We
write $H=H_\partial \times \widehat H$ by choosing a splitting.
Choosing an arbitrary basis of $H_\partial$ and a symplectic basis of
$\widehat H=H_1(\widehat\Sigma)$, $\Aut^*(H)$ consists of all matrices
$P$ of the form
\[
\begin{bmatrix}
  \id_{n-1} & * \\
  0 & P_0
\end{bmatrix}
\]
with $P_0\in \mbox{Sp}(2g,\Z)$ \cite[Remark~2.4]{GS08}. Note that any
$P\in \Aut^*(H)$ is indeed realized by a homology cylinder, in fact
there exists $\varphi\in \mgn$ such that the induced action on
$H_1(\sgn)$ is given by $P$ (e.g., see \cite[Proposition~7.3]{FM09}).

\subsection{Product formulas for torsion}

Each $\varphi\in \aut(H)$ induces an automorphism of $\zh$, which we
also denote by $\varphi$.  In particular, a homology cylinder $M$ over
$\Sigma$ gives rise to $\varphi(M)\in \Aut(\zh)$.  We can now
formulate the following proposition (e.g., see \cite[Section~7]{Mi66}
and \cite[Proposition~6.6]{Sa08} for related results).

\begin{proposition}\label{prop:deltasum}
  Let $M=(M,i_+,i_-)$ and $N=(N,j_+,j_-)$ be homology cylinders
  over~$\Sigma$. Then
  \[
  \tau(M\cdot N)\doteq \tau(M) \cdot \varphi(M)(\tau(N)).
  \]
\end{proposition}

Note that $Q(H)^\times$ is a multiplicative abelian group. The
action of $\Aut^*(H)$ on $H$ extends to an action of $\Aut^*(H)$ on
$Q(H)^\times$, and we can thus form the semidirect product
$\Aut^*(H)\ltimes Q(H)^\times$. We can now formulate the following
corollary:

\begin{corollary}\label{cor:metabmap}
The following is a well-defined homomorphism of monoids
\[ \begin{array}{rcl} \cgn &\to & \Aut^*(H)\ltimes Q(H)^\times \\
M&\mapsto & (\varphi(M),\tau(M)).\end{array}\]
\end{corollary}

Since we are mostly interested in abelian quotients of $\cgn$ we
will now show that $\tau$ also gives rise to a homomorphism to an
abelian group. We define $A(H)$ to be the subgroup of $Q(H)^\times$
generated by the following set
\[
\{\pm h \cdot p^{-1}\cdot \varphi(p) \mid h\in H\text{, } p\in
Q(H)^\times\text{, and }\varphi\in \Aut^*(H)\}.
\]
We write $\AH=A(H)$ when $H$ is clearly understood from the context.
The following is an immediate consequence of
Proposition~\ref{prop:deltasum}:

\begin{corollary}\label{corollary:monoid-homomorphism}
  The torsion invariant gives rise to a monoid homomorphism
  \[
  \tau\colon \cgn \to Q(H)^\times/\AH.
  \]
\end{corollary}

\begin{proof}[Proof of Proposition~\ref{prop:deltasum}]
  We write $\S_\pm=i_\pm(\S)$ and $\L_\pm=j_\pm(\S)$.  Let
  $W=M\cup_{\S_-=\L_+} N$, where $i_-\circ(j_+)^{-1}$ is the gluing
  map.  We view $M$, $N$, $\Sigma_-$ (=$\Lambda_+$) as subspaces
  of~$W$.  A Mayer-Vietoris argument shows that the inclusion map
  induces an isomorphism $H_1(M)\to H_1(W)$.  Let $F=H_1(M)$.  We
  equip $W$ with the map $H_1(W) \cong H_1(M) = F$ induced by the
  inclusion.  Restricting this, we equip $\Sigma_-$ ($=\Lambda_+$),
  $M$, $N$ with maps into~$F$.  For notational convenience, we denote
  $Q=Q(F)$, the quotient field of~$\Z[F]$.  We have the following
  short exact sequence of cellular homology:
  \[
  0\to C_*(\S_-;Q)\to C_*(M,\S_+;Q)\oplus C_*(N;Q)\to C_*(W,\S_+;Q)\to
  0.
  \]
Here the map $C_*(\S_-;Q)\to C_*(N;Q)$ is given by $-(j_+)_*\circ
(i_-)^{-1}_*$ and the map  $C_*(\S_-;Q)\to C_*(M,\S_+;Q)$ is given
by the inclusion map. Let $\BB_*$ be a basis for the graded
$Q$--module  $H_*(\S_-;Q)$.  It follows from the long exact homology
sequence that the inclusion map induces an isomorphism
$H_*(\S_-;Q)\to H_*(N;Q)$.  We equip $H_*(N;Q)$ with the basis
$\CC_*$ given by the image of $\BB_*$.  We consider the resulting
torsions $ \tau_{\BB_*}(\S_-;Q)$, $\tau(M,\S_+;Q)$,
$\tau_{\CC_*}(N;Q)$, and $\tau(W,\S_+;Q)$. By applying
\cite[Theorem~3.2]{Mi66} to the above short exact sequence, the
torsions satisfy the following:
\[
  \tau(M,\S_+;Q)\cdot \tau_{\CC_*}(N;Q)\doteq \tau(W,\S_+;Q)\cdot
  \tau_{\BB_*}(\S_-;Q).
\]
We equip $H_*(\L_+;Q)$ with the basis $\DD_*=(j_+)_*\circ
(i_-)_*^{-1}(\BB_*)$.  Similarly we have
\[
  \tau_{\CC_*}(N;Q)\doteq \tau(N,\L_+;Q)\cdot \tau_{\DD_*}(\L_+;Q).
\]
From the above equations and the tautology
$\tau_{\DD_*}(\L_+;Q)=\tau_{\BB_*}(\S_-;Q)$, it follows that
\[
  \tau(W,\S_+;Q) \doteq \tau(M,\S_+;Q) \cdot \tau(N,\L_+;Q).
\]
We now apply $(i_+)_*^{-1}\colon F\to H$ to the above equality. From
the commutative diagram
\[
  \xymatrix{
    H_1(N) \ar[r]^\cong & H_1(W) & H_1(M)=F \ar[l]_\cong\\
    & H \ar[ul]^{(j_+)_*} \ar[ur]_{(i_-)_*} \ar[r]_{\varphi(M)} & H \ar[u]_{(i_+)_*}
  }
\]
it follows that our $H_1(N)\to F$ composed with $(i_+)_*^{-1}$ gives
$\varphi(M)(j_+)_*^{-1}$.  Therefore we have
\[
  \tau(M\cdot N)\doteq \tau(M) \cdot \varphi(M)(\tau(N)). \qedhere
\]
\end{proof}

\begin{remark}
  Let $H$ be a finitely generated free abelian group.  Given a
  non-zero polynomial $p\in \zh$ we denote by $m(p)\in \R_{>0}$ its
  Mahler measure (e.g., see \cite{EW99}, \cite{SW02}, and
  \cite{SW04}).  It is well-known that given $p,q\in \zh$ and $h\in H$
  we have $m(p\cdot q)=m(p)\cdot m(q)$ and $m(\pm h\cdot p)=m(p)$.
  Furthermore given any $\varphi \in \aut(H)$ we have
  $m(\varphi(p))=m(p)$ (e.g., see \cite[Section~3]{EW99}).  It follows
  that $M\mapsto m(\tau(M))$ defines a monoid homomorphism $\cgn \to
  \R_{>0}$. Using this homomorphism and using the examples of Section
  \ref{section:tyingknot} one can reprove Theorem 2.4 in~\cite{GS09}.
\end{remark}

Now let $M$ be a homology cylinder over a surface $\S=\S_{g,n}$ and
let $M'$ be a homology cylinder over a surface $\S'=\S'_{g',n'}$.
Assume that $n,n'>0$ and fix a boundary component $c$ in $\S$ and fix
a boundary component $c'$ in $\S'$.  Recall that gluing now gives us a
homology cylinder $M\cup_{c,c'} M'$ over $\S\cup_{c=c'} \S'$.

\begin{proposition}\label{prop:tauunion}
  Denote by $i$ and $i'$ the inclusion induced maps of $H_1(\S)$ and
  $H_1(\S')$ into $H_1(\Sigma\cup_{c=c'}\Sigma')$, respectively.  Then
  the following holds:
  \[
  \tau(M\cup_{c,c'} M')=i(\tau(M))\cdot i'(\tau(M')).
  \]
\end{proposition}

\begin{proof}
  Write $H=H_1(\Sigma\cup_{c=c'}\Sigma')$.  By the Mayer-Vietoris
  formula for torsion we have
  \[
  \tau(M\cup_{c,c'} M')=\tau(M;Q(H))\cdot \tau(M';Q(H)).
  \]
  It is seen easily that $\tau(M;Q(H))=i(\tau(M))$ and
  $\tau(M';Q(H))=i'(\tau(M'))$ from the definitions.
\end{proof}

\subsection{Torsion and homology cobordisms}

Let $H$ be a free abelian multiplicative group. We equip $\zh$ with
the standard involution of a group ring, i.e. $\ol{h}=h^{-1}$ for
$h\in H$ and extend it to $Q(H)$ by setting $\ol{p\cdot
  q^{-1}}=\ol{p}\cdot \ol{q}^{-1}$.

The next theorem can be viewed as a generalization of the classical
Fox-Milnor theorem \cite{FoM66} that the Alexander polynomial of a
slice knot factors as $\Delta_K(t)\doteq f(t)f(t^{-1})$ for some
$f(t)\in \zt$.

\begin{theorem} \label{thm:homcob}
Let $M=\mi$ and $N=\nj$ be homology cylinders over a surface
$\Sigma$ which are homology cobordant. Write $H=H_1(\Sigma)$. Then
\[
  \tau(M)\doteq \tau(N) \cdot q\cdot \bar{q} \in Q(H)^\times
\]
for some $q\in Q(H)^\times$.
\end{theorem}

\begin{proof}
Let $W$ be a homology cobordism between $M$ and $N$.  View $\S_+=
i_+(\S)=j_+(\S)$ as a subspace of~$W$. Note that $H_1(\S_+)\to
H_1(W)$ is an isomorphism. We now equip $W$ with the homomorphism
\[
  H_1(W) \xrightarrow{\cong}  H_1(\S_+)
  \xrightarrow{(i_+)_*^{-1}} H.
\]
By the above we have $H_1(W,\Sigma_+)=0$.  Therefore it follows from
Lemma~\ref{lem:alsozero} that $H_1(W,\S_+;Q(H))=0$ and we can
therefore consider $\tau(W):=\tau(W,\S_+;Q(H)) \in Q(H)$.  Also note
that $H_*(W,M)=0$ and $H_*(W,N)=0$.  We can hence also think of the
torsion invariants $\tau(W,M):=\tau(W,M;Q(H))$ and
$\tau(W,N):=\tau(W,N;Q(H))$. From the following short exact sequence
of acyclic chain complexes
$$
  0\to C_*(M,\S_+;Q(H))\to C_*(W,\S_+;Q(H)) \to C_*(W,M;Q(H))\to 0,
$$
we obtain $\tau(W) \doteq \tau(M)\cdot \tau(W,M)$. Similarly,
$\tau(W) \doteq \tau(N) \cdot \tau(W,N)$. By duality we have
$\tau(W,M)\doteq \overline{\tau(W,N)}^{-1}$ (e.g., see \cite{Mi66},
\cite{KL99} and \cite{CF10}).  Hence
$$
\tau(M) \doteq \tau(N) \cdot \tau(W,N)\cdot \overline{\tau(W,N)},
$$
and the theorem follows.
\end{proof}

Define
\[
N(H)=\{\pm h \cdot q \cdot \bar q \mid q \in Q(H)^\times,\, h\in H\}.
\]
Obviously $N(H)$ is a subgroup of~$Q(H)^\times$.  As we do with
$\AH=A(H)$, we often write $\NH=N(H)$.

The following corollary is now an immediate consequence of
Corollary~\ref{cor:metabmap} and Theorem~\ref{thm:homcob}. This
corollary should be compared to \cite[Theorem~5.1]{Mo08}.

\begin{corollary}\label{cor:metabmap2}
The map $\cgn\to \Aut^*(H)\ltimes Q(H)^\times$ defined in Corollary
\ref{cor:metabmap} descends to a group homomorphism
\[  \hgn \to \Aut^*(H)\ltimes Q(H)^\times/\NH.\]
\end{corollary}

Since we are mostly interested in abelian quotients of $\hgn$ we
will for the most part work with the following corollary, which  is
an immediate consequence of Theorem~\ref{thm:homcob}.

\begin{corollary}\label{corollary:group-homomorphism}
  The torsion invariant gives rise to a group homomorphism
  \[
  \tau\colon \hgn \to Q(H)^\times/\ANH,
  \]
  where $H=H_1(\sgn)$.  This descends to a homomorphism of the
  quotient of $\hgn$ modulo the normal subgroup $\langle\mgn\rangle$
  generated by the mapping class group~$\mgn$.
\end{corollary}

In later sections, we will show that the image of this torsion
homomorphism is large.  Indeed, it has a quotient isomorphic to
$(\Z/2)^\infty$ if $b_1(\Sigma_{g,n})=\rank H >0$, and has a
quotient isomorphic to $\Z^\infty$ if either $n>2$, or $n=2$ and
$g>0$.

\subsection{Symmetry and asymmetry of torsion}
\label{subsection:symmetry-asymmetry}

It is well-known that the Alexander polynomial $\Delta_K(t)$ of a knot
$K$ is symmetric, i.e. $\Delta_K(t)\doteq \Delta_K(t^{-1})$.
Furthermore, the Alexander polynomial of any closed $3$--manifold or
any $3$--manifold with toroidal boundary is symmetric (e.g.,
see~\cite{Tu86}). In this section we will study the symmetry
properties of the torsion of homology cylinders.

We start out with the following observation.

\begin{lemma}
  Let $\S$ be a surface which is either an annulus or a surface of
  genus one with at most one boundary component.  Let $M$ be a
  homology cylinder over $\S$.  Then $\tau(M)\doteq \ol{\tau(M)}$.
\end{lemma}

\begin{proof}
  First assume that $\S$ is either an annulus or a torus.  We write
  $H=H_1(M)$, $Q=Q(H)$ and consider the torsion $\tau(M,\S_+;Q)$
  corresponding to $\pi_1(M)\to H_1(M)=H$.  Clearly it suffices to
  show that $\tau(M,\S_+;Q)$ is symmetric.  First note that $\S_+$ has
  Euler characteristic zero. It can be seen easily that
  $H_*(\S_+;Q)=0$, in particular its torsion $\tau(\S_+;Q)$ is
  defined.  Also, from this it follows that $M$ is $Q$-acyclic so that
  $\tau(M;Q)$ is defined.  It is well-known that the torsion of a
  torus is trivial and for an annulus it equals $(1-t)^{-1}$, where
  $t$ is a generator of the first homology group. It follows in
  particular that $\tau(\S_+;Q)\doteq \tau(\S_-;Q)$.  From the long
  exact sequence of torsion corresponding to the pair $(M,\S_\pm)$ it
  now follows that $\tau(M,\S_\pm;Q)=\tau(M;Q)\cdot
  \tau(\S_\pm;Q)^{-1}$.  Also, from duality for torsion we have
  $\tau(M,\S_+;Q)\doteq \ol{\tau(M,\S_-;Q)}$.  Combining these
  equalities, it follows that
  \begin{align*}
    \tau(M,\S_+;Q)&=\tau(M;Q)\cdot \tau(\S_+;Q)^{-1}\\
    &=\tau(M;Q)\cdot \tau(\S_-;Q)^{-1}\\
    &=\tau(M,\S_-;Q)\\
    &\doteq \ol{\tau(M,\S_+;Q)}
  \end{align*}
  as desired.

  Now assume that $\S$ is a surface of genus one with exactly one
  boundary component.  Let $\S'$ be a disk and $T=\S\cup \S'$ the
  result of gluing $\S$ and $\S'$ along their boundary components. Let
  $M'=\S'\times [0,1]$.  Note that $H_1(\S)\to H_1(T)$ is an
  isomorphism.  It now follows immediately from Proposition
  \ref{prop:tauunion} that $\tau(M)=\tau(M\cup M')$, in particular
  $\tau(M)$ equals the torsion of a homology cylinder over a torus,
  which is symmetric as we saw above.
\end{proof}

In general torsion is not symmetric though. To our knowledge this phenomenon
was first observed in \cite[Example~8.5]{FJR09}
in the context of sutured manifolds and rational homology cylinders.
More examples will be given in Section \ref{section:construction-computation}.

On the other hand recall that in order to define the torsion
homomorphism in Corollaries~\ref{corollary:monoid-homomorphism} and
\ref{corollary:group-homomorphism}, we thought of torsion invariants
modulo $\AH=A(H)$, i.e., up to the action of $\Aut^*(H)$.  When the
number of boundary components $n\le 1$, $\Aut^*(H)$ is simply the
group of automorphisms on $H$ preserving the intersection pairing
on~$\Sigma$.  In particular, the map $\varphi(h)=h^{-1}$ is in
$\Aut^*(H)$. Therefore $p\cdot \AH = \varphi(p)\cdot \AH = \bar p \cdot
\AH$ in $Q(H)^\times / \AH$ for any~$p$.  In other words, modulo $\AH$,
everything is symmetric:

\begin{lemma}\label{lemma:symmetry-up-to-aut*}
  Suppose $\Sigma$ is a surface with at most one boundary components.
  Then for any homology cylinder $M$ over $\Sigma$, we have $\tau(M)=
  \overline{\tau(M)}$ in $Q(H)^\times/\AH$.
\end{lemma}

Therefore the only remaining case is when either $n>2$, or $n=2$ and
$g>0$.  We will show that in these cases the torsion invariant
$\tau(M)$ is in general asymmetric \emph{even modulo $\AH=A(H)$}.

In the next section we will consider general methods to construct
homology cylinders and to compute their torsion, which will be used to
illustrate the asymmetry of torsion.

\section{Constructions and computation}
\label{section:construction-computation}

In this section we will compute torsion for various homology
cylinders. The examples illustrate the computation of torsion and
they will also be used later to prove Theorems \ref{mainthm} and
\ref{mainthm2}.

\subsection{Handle decomposition}

For any homology cylinder $(M,i_+,i_-)$ over $\Sigma$, the pair
$(M,\Sigma_+)$ admits a handle decomposition without $0$-- and
$3$--handles.  A handle decomposition of a pair $(M,\Sigma_+)$ is
given as submanifolds
\[
\Sigma_+\times[0,1] = M_0 \subset M_1 \subset M_2 = M
\]
where $M_i$ is obtained by attaching $i$-handles to~$M_{i-1}$ for
$i=1,2$.  $M$ is a homology cylinder if and only if the numbers of
$1$-handles and $2$-handles are equal, say $r$, and the boundary map
\[
\partial\colon C_2(M,\Sigma_+)=\Z^r \to C_1(M,\Sigma_+) = \Z^r
\]
is invertible.

From the definition, the torsion $\tau(M)=\tau(M,\Sigma_+;Q(H))$ is equal
to the determinant of the $r\times r$ matrix $A$ over $\Z[H]$ which
represents the $\Z[H]$-coefficient boundary map
\[
\partial\colon C_2(M,\Sigma_+;\Z[H])=\Z[H]^r \to C_1(M,\Sigma_+;\Z[H])
= \Z[H]^r
\]
with respect to the bases given by $1$-- and $2$--handles.  When a
handle decomposition is explicitly given, the boundary map can be
effectively computed using a standard method.

This is often useful in computing the torsion of a homology cylinder
with a given handle decomposition.  For readers who are not familiar
with this type of computation, it can be described best by a
detailed example.  The construction and computation below will also
be used later, to show the existence of nontrivial homomorphisms of
$\hgn$ onto~$\Z$.

\subsection{An example of asymmetric torsion}
\label{subsection:exmaple-handle-decomp}

Let $\Sigma=\Sigma_+=\Sigma_{g,n}$ with $g>0$ and $n=2$.  Let $M_1$ be
the $3$--manifold obtained by attaching one $1$--handle to
$\Sigma\times[0,1]$.  The boundary of $M_1$ is the union of $\Sigma$,
$\partial \Sigma\times[0,1]$, and a genus $g+1$ surface, say
$\Sigma_1$, which has the same boundary as $\Sigma$.  We attach one
$2$--handle to $M_1$ along the simple closed curve $\alpha$ on
$\Sigma_1$ as illustrated in
Figure~\ref{figure:example-attaching-curve}.  The point $*$ is the
basepoint and the gray band
\definecolor{mygray}{gray}{.8}%
{\setbox0=\hbox{\fbox{$a$}}\kern\wd0
  \hbox{\vbox{\setbox1=\hbox{$a$}\hrule width 2em height .4pt
      \begin{color}{mygray}\hrule width 2em height\ht1\end{color}%
      \hrule width 2em height .4pt}%
    \kern-2em\kern-\wd0\copy0\kern2em\copy0}}
represents $a$ parallel strands.  We denote the resulting
$3$--manifold by $M=M(a)$.

\begin{figure}[t]
  \begin{center}
    \labellist
    \small\hair 2pt
    \pinlabel {$x^*$} at 145 63
    \pinlabel {$y_1^*$} at 272 15
    \pinlabel {$y_2^*$} at 272 95
    \pinlabel {$t^*$} at 261 228
    \pinlabel {$\alpha$} at 120 287
    \pinlabel {$\beta$} at 164 239
    \pinlabel {\rotatebox{90}{$a$}} at 103 194
    \pinlabel {\rotatebox{90}{$a$}} at 102 169
    \endlabellist
    \includegraphics[scale=.9]{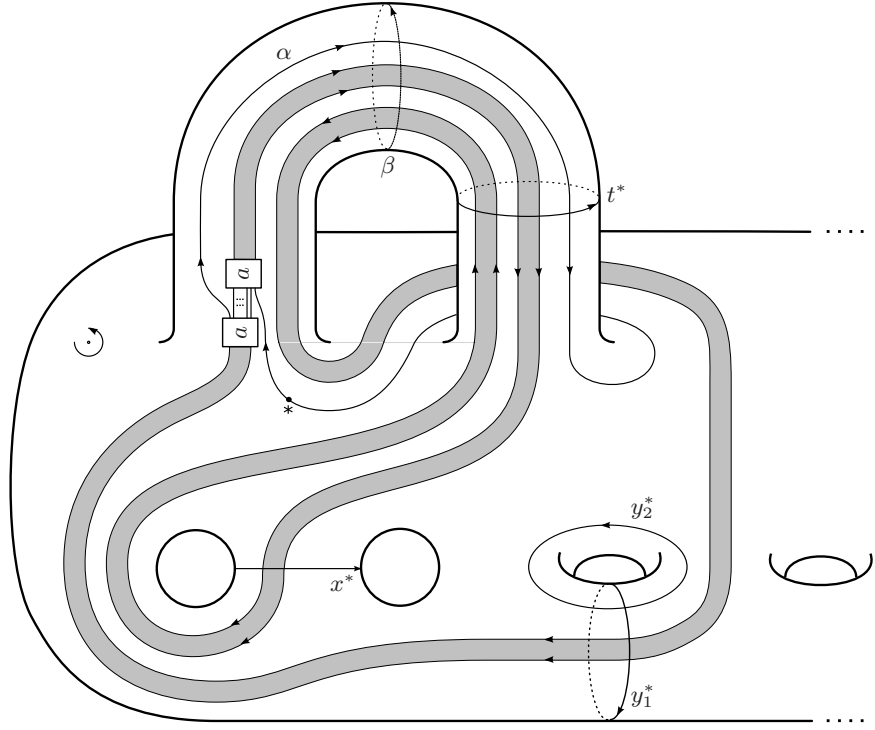}
  \end{center}
  \caption{Attaching circle $\alpha$ on $\Sigma_1$}
  \label{figure:example-attaching-curve}
\end{figure}

We have $H_1(M_1)=\Z\times H$ where $\Z$ is generated by the
$1$--handle.  Recall that we can find a splitting $H=H_\partial \times
\widehat H$ where $H_\partial$ is the image of $H_1(\partial\Sigma)$
and $\widehat H=H_1(\Sigma$ with capped-off boundary$)$.  In our case
$H_\partial$ and $\widehat H$ have rank $1$ and $2g$, respectively.
We choose generators $t$ of $\Z$, $x$ of $H_\partial$, and $y_i$
($i=1,\ldots,2g$) of $\widehat H$ which lie on $\Sigma_1$ and which
are dual to the curves $t^*$, $x^*$, $y_i^*$ illustrated in
Figure~\ref{figure:example-attaching-curve} with respect to the
intersection pairing on~$\Sigma_1$.  That is, $t$ is represented by a
loop which is disjoint to $x^*$, $y_i^*$ and has intersection numbers
$+1$ with $t^*$ on $\Sigma_1$, and similarly for other generators.
Given a curve on $\Sigma_1$, its homology class in $H_1(M_1)$ can be
determined easily by counting the algebraic intersections with the
dual curves.

For the curve $\alpha$, we have $[\alpha]=(t,
x^a_{\vphantom{1}}y_1^a) \in \Z\times H = H_1(M_1)$.  Thus the
homology class of $\alpha$ in $H_1(M_1,\Sigma_+)=\Z$ is $t$, and so
the boundary map
\[
\partial\colon C_2(M,\Sigma_+)=\Z \to C_1(M,\Sigma_+) = \Z
\]
is the identity.  Therefore our $M$ is a homology cylinder.

Now we compute the $\Z[H]$--coefficient boundary map.  Since the
attaching circle gives a relation $t=(xy_1)^{-a}$, the generator $t$
is sent to $(xy_1)^{-a}$ under $H_1(M)
\xrightarrow{\raise-.5ex\hbox{$\scriptstyle\smash{\cong}$}}
H_1(\Sigma_+)=H$.  The $\Z[H]$--valued intersection number
$\lambda(\alpha,\beta)$ of $\alpha$ with the belt circle $\beta$ of
the $1$--handle is given by the following formula: for a point $u$ on
$\alpha$, let $\alpha_u$ be the oriented initial segment of $\alpha$
from the basepoint $*$ to~$u$.  Then
\[
\lambda(\alpha,\beta) = \sum_{u\in \alpha \cap \beta}
\Bigg(\epsilon(u) \cdot \bigg(\prod_{\vphantom{\{}v\in \alpha_u \cap
t^*} (xy_1)^{-a\epsilon(v)}\bigg) \cdot \bigg(\prod_{c\in \{x,y_1,
y_2,\ldots, y_{2g}\}} \prod_{\vphantom{\{}w\in \alpha_u \cap c^*}
c^{\epsilon(w)} \bigg) \Bigg).
\]
Here, $\epsilon(u)$ is the usual sign of $u\in \alpha\cap \beta$.
Precisely, $\epsilon(u)=1$ if $(d\alpha/dt, d\beta/dt)$ at $u$ agrees
with the orientation of~$\Sigma_1$, and $-1$ otherwise.
The numbers $\epsilon(v)$, $\epsilon(w)$ are given similarly.

From Figure~\ref{figure:example-attaching-curve}, it can be seen that
\begin{align*}
  \lambda(\alpha,\beta) & = \big(1+(-1)\cdot x\big) +
  \big( 1\cdot(xy)+(-1)\cdot(x^2y) \big) + \cdots \\
  & \qquad + \big( 1\cdot(x^{a-1}y^{a-1})+(-1)\cdot(x^ay^{a-1}) \big)
  + 1\cdot x^ay^a \\
  & = 1+(y-1)x+y(y-1)x^2+\cdots+y^{a-1}(y-1)x^a
\end{align*}
where $y=y_1$.  The $\Z[H]$--coefficient boundary map is the $1\times
1$ matrix $[\lambda(\alpha,\beta)]$.  Therefore $\tau(M(a)) =
\lambda(\alpha,\beta)$.

It is obvious that $\tau(M(a))$ is asymmetric in the sense that
$\tau(M(a)) \not\doteq \overline{\tau(M(a))}$.  In
Example~\ref{example:non-self-dual-element}, we will show that
$\tau(M(a))$ and $\overline{\tau(M(a))}$ are distinct even
modulo~$\AH=A(H)$.

\subsection{Tying in a string link}

In this
subsection we describe an operation that modifies a homology cylinder
by ``tying in'' a string link, and investigate its effect on the
torsion invariant.

First we recall the definition of a string link.  Fix $m$ points $p_1,
\ldots, p_m$ in the interior of the disk $D^2$.  An $m$--component
string link $\beta$ is the disjoint union of $m$ properly embedded
disjoint oriented arcs $\beta_i$ ($i=1,\ldots,m$) in $D^2\times[0,1]$
from $p_i\times 0$ to $p_i \times 1$.  We denote the exterior of a
string link $\beta$ in $D^2\times[0,1]$ by~$E_\beta$.  The string link
$\bigcup p_i \times [0,1] \in D^2\times[0,1]$ is called the trivial
string link.  The exterior of the trivial string link is denoted
by~$E_0$.  Note that if $\beta$ is framed, then there is a canonical
identification of $\partial E_\beta$ and~$\partial E_0$ which
restricts to the identity on $D^2\times \{0,1\} \cap E_0$.

\begin{remark}
  Including Theorem~\ref{theorem:string-link-sum-formula}, all results
  in this section hold for string links in homology $D^2\times[0,1]$
  as well.
\end{remark}

\subsubsection*{Torsion of string links}

We recall the definition of the torsion invariant for string links
(cf.\ \cite[Definition~6.8]{KLW01}).  Let $X=(D^2\times 0) \cap E_0$,
a subspace of $D^2 \times[0,1]$.  Suppose $\beta$ is a string link and
$\phi\colon H_1(E_\beta) \to H$ is a homomorphism to a free abelian
group~$H$.  Since $(E_\beta,X)$ is $\Z$--acyclic, it follows from
Lemma \ref{lem:alsozero} that $(E_\beta,X)$ is $\Z[H]$--acyclic.  We
denote by $\tau^\phi_\beta \in \zh^\times$ the torsion of the cellular
chain complex $C_*(E_\beta,X;\Z[H])$.  When $\phi$ is clearly
understood from the context we write $\tau_\beta=\tau^\phi_\beta$.  As
usual, $\tau_\beta$ is well-defined up to multiplication by $\pm h$
($h\in H$).

Note that for any string link~$\beta$, $E_\beta$ can be viewed as a
homology cylinder over the surface~$X$.

\subsubsection*{Formula for string link tying}
Now consider a homology cylinder $(M,i_+,i_-)$ on $\Sigma =
\Sigma_{g,n}$ with $H=H_1(\Sigma)$ and an embedding $f\colon E_0 \to
\inte(M)$.  For any framed string link $\beta$, define
\[
M(f,\beta) = \big(M-f(\inte (E_0))\big)
\mathbin{\mathop{\cup}\limits_{f(\partial E_0)=\partial E_\beta}}
E_\beta.
\]
Since $E_\beta$ has the same homology as $E_0$, it follows that
$(M(f,\beta),i_+,i_-)$ is a homology cylinder.  We say that
$M(f,\beta)$ is obtained from $M$ by tying the string link~$\beta$.

The map $f$ gives rise to a homomorphism
\[
H_1(E_\beta) \xrightarrow{\cong} H_1(X) \xrightarrow{\cong} H_1(E_0)
\xrightarrow{f_*} H_1(M)\xleftarrow{} H_1(\S_+)
\xrightarrow{(i_+)_*^{-1}} H_1(\Sigma)=H
\]
induced by the inclusions. We denote the resulting torsion invariant
of $\beta$ by~$\tau_\beta^f$.

\begin{theorem}\label{theorem:string-link-sum-formula}
  The torsion invariant of the homology cylinder $M(f,\beta)$ is given
  by $\tau(M(f,\beta))\doteq \tau(M) \cdot \tau_\beta^f$.
\end{theorem}

\begin{proof}
  We write $\S_+=i_+(\S)$.  Let $M' = M-f(\inte (E_0))$.  Choose bases
  of $H_*(X;Q(H))$ and $H_*(M',\Sigma_+;Q(H))$.  Since $(E_\beta,X)$
  is $\Z$--acyclic, it is $Q(H)$--acyclic, and therefore for any string
  link $\beta$ (including the trivial one) our basis of $H_*(X;Q(H))$
  determines a basis of $H_*(E_\beta;Q(H))$ via the inclusion-induced
  map.  Define the torsion invariants $\tau(X)$, $\tau(E_\beta)$,
  $\tau(M',\Sigma_+) \in Q(H)^\times$ (up to the usual ambiguity)
  using these bases.

  Consider the following short exact sequence:
  \begin{multline*}
    0 \to C_*(\partial E_0;Q(H)) \to C_*(M',\Sigma_+;Q(H)) \oplus
    C_*(E_\beta;Q(H)) \\
    \to C_*(M(f,\beta),\Sigma_+;Q(H)) \to 0.
  \end{multline*}
  Since $(M(f,\beta),\Sigma_+)$ is $Q(H)$-acyclic, the bases we have
  chosen give rise to a basis of $H_*(\partial E_0;Q(H))$.  Define
  $\tau(\partial E_0) \in Q(H)^\times$ using this basis.  By
  \cite[Theorem 3.2]{Mi66}, we have
  \[
  \tau(M',\Sigma_+)\cdot \tau(E_\beta) \doteq \tau(\partial E_0)\cdot
  \tau(M(f,\beta)).
  \]
  Considering the special case of a trivial string link $\beta$, we
  obtain
  \[
  \tau(M',\Sigma_+)\cdot \tau(E_0) \doteq \tau(\partial E_0)\cdot \tau(M).
  \]
  Combining these, we have
  \[
  \tau(M(f,\beta)) \doteq \tau(M) \cdot \tau(E_\beta) \cdot \tau(E_0)^{-1}.
  \]
  From the short exact sequence
  \[
  0\to C_*(X;Q(H)) \to C_*(E_\beta;Q(H)) \to C_*(E_\beta,X;Q(H))\to 0
  \]
  with $(E_\beta,X)$ acyclic, we obtain $\tau(E_\beta) \doteq \tau(X) \cdot
  \tau_\beta^f$.  Note that $E_0=X\times [0,1]$, hence
  $\tau(E_0)=\tau(X)$.   The desired formula now follows immediately.
\end{proof}

\subsubsection*{Special case: tying in a  knot}
\label{section:tyingknot}

We will now study the case $m=1$.  Note that a string link $\beta$
with one component gives us a knot $K$ in~$S^3$.
Furthermore, for any embedding $f\colon E_0\to \inte(M)$, it can be
seen easily that the torsion $\tau_\beta^f$ is equal to $\Delta_K(h)$,
where $\Delta_K(t)$ is the Alexander polynomial of $K$ and $h$ is the
image of the generator of $H_1(E_0)=\Z$ under the map $H_1(E_0) \to H$
induced by~$f$.  Therefore, by
Theorem~\ref{theorem:string-link-sum-formula}, the torsion of
$M(f,\beta)$ can be described in terms of the Alexander polynomial
of~$K$.  The following is a special case of this, which gives a group
homomorphism of the (smooth or topological) concordance group $\CC$ of
knots in $S^3$ into~$\hgn$. Recall that $E_0=X\times[0,1]$ and $X=E_0
\cap (D^2\times 0)$.

\begin{proposition}\label{prop:embedc}
  Let $g,n\geq 0$.  We write $\S=\sgn$ and $M=\S\times [0,1]$.  Let
  $\iota\colon X\to \inte(\S)$ be an embedding, and let $f\colon E_0
  \to\inte(M)$ be the embedding $f(x,t)=\iota(x,
  t/2+1/4)$. Then the
  assignment $K\mapsto M(f,K)$ descends to a group homomorphism
  \[
  \CC\to \hgn.
  \]
  Furthermore
  \[
  \tau(M(f,K))=\Delta_K(h)
  \]
  where $h$ is the image of the generator of $H_1(X)\cong \Z$ under
  $\iota_*\colon H_1(X) \to H$.
\end{proposition}

\begin{proof}
  It is straightforward to verify that $M(f,K_1\# K_2) \cong
  M(f,K_1)\cdot M(f,K_2)$ and that the assignment $K\mapsto M(f,K)$ in
  fact descends to a group homomorphism $\CC\to \hgn$.  The conclusion
  on torsion has already been proven in the paragraph above
  Proposition~\ref{prop:embedc}.
\end{proof}

\section{Epimorphisms onto infinitely generated abelian groups}
\label{section:homomorphisms}

In this section we will construct epimorphisms of $\hgn$ onto
nontrivial abelian groups. This will give a proof of
Theorems~\ref{mainthm} and \ref{mainthm2}.  Throughout
Section~\ref{section:homomorphisms}, we fix $g,n\geq 0$ and write
$\Sigma=\sgn$ and $H=H_1(\S)$.

\subsection{Algebraic structure of the torsion group}
\label{section:algeraic structure}

Recall that our torsion invariant lives in the multiplicative abelian
group~$Q(H)^\times/\ANH$.  In this subsection we investigate the
algebraic structure of $Q(H)^\times/\ANH$.  We will think of
symmetric/asymmetric parts of $Q(H)^\times/\ANH$, and define certain
$\Z/2$ and $\Z$-valued homomorphisms which form sets of complete
invariants of these parts, respectively.  To be more precise, let
\[
Q(H)^{sym} = \{p\in Q(H)^\times \mid p=\bar p \text{ in
}Q(H)^\times/\AH\}.
\]
Note that $\ANH\subset Q(H)^{sym}$.  There is an exact
sequence
\[
1 \to \frac{Q(H)^{sym}}{\ANH} \to \frac{Q(H)^\times}{\ANH}
\to \frac{Q(H)^\times}{Q(H)^{sym}} \to 1
\]
which can be viewed as a decomposition of $Q(H)^\times/AN$ into
symmetric and asymmetric parts.

Recall that $H_\partial$ is the image of $H_1(\partial \S)$ in~$H$ and
\[
\Aut^*(H) = \{\varphi\in \Aut(H) \mid \text{$\varphi$ fixes
  $H_\partial$ and preserves the intersection form of $\Sigma$}\}.
\]
We define an equivalence relation $\sim$ on $\zh- \{0\}$ by $p \sim
q$ if $p\doteq \varphi(q)$ for some $\varphi\in \Aut^*(H)$.  Note that
if $p\sim q$ then $p$ and $q$ represent the same element in
$Q(H)^\times/\AH$. From now on we say $p$ is \emph{self-dual} if
$p\sim \bar p$.

Recall that $\Z[H]$ is a unique factorization domain, so that for each
$p\in Q(H)$ and irreducible $\lambda\in \Z[H]$, we can think of the
exponent of $\lambda$ in the factorization of~$p$.  (The exponent is
an integer and may be negative.)  For an irreducible element $\lambda$
in $\zh$, we define a function $e_\lambda\colon Q(H)^\times \to \Z$ as
follows: Given $p\in Q(H)$, $e_\lambda$ is the sum of exponents of
distinct irreducible factors $\mu$ of $p$ such that $\mu\sim \lambda$.
(As usual, irreducible factors are distinguished up to multiplication
by a unit in~$\zh$.)

\begin{proposition}\label{prop:complete}
  \leavevmode\Nopagebreak
  \begin{enumerate}
  \item If $\lambda$ is a self-dual irreducible element in $\zh$, then
    the map
    \[
    \Psi_\lambda\colon Q(H)^{sym}/\ANH \to \Z/2
    \]
    defined by $\Psi_\lambda(p\cdot \ANH)=e_\lambda(p)+2\Z$ is a
    surjective group homomorphism.  Furthermore,
    \[
    \Psi = \bigoplus_{[\lambda]}\Psi_\lambda \colon Q(H)^{sym}/\ANH \to
    \bigoplus_{[\lambda]} \Z/2,
    \]
    is an isomorphism, where $[\lambda]$ runs over the equivalence
    classes of self-dual irreducible~$\lambda$.
  \item If $\mu$ is a non-self-dual irreducible element in $\zh$, then the
    map
    \[
    \Theta_\mu\colon Q(H)^\times /Q(H)^{sym} \to \Z
    \]
    defined by $\Theta_\mu(p\cdot Q(H)^{sym})=e_\mu(p)-e_{\bar\mu}(p)$
    is a surjective group homomorphism.  Furthermore
    \[
    \Theta = \bigoplus_{\{[\mu],[\bar\mu]\}} \Theta_\mu \colon
    Q(H)^\times / Q(H)^{sym}
    \to \bigoplus_{\{[\mu],[\bar\mu]\}} \Z,
    \]
    is an isomorphism,
    where $\{[\mu],[\bar\mu]\}$ runs over the unordered pairs of
    equivalence classes of non-self-dual irreducible $\mu$ and its
    involution $\bar\mu$.
  \end{enumerate}
  Consequently, $Q(H)^\times/\ANH$ is isomorphic to
  $\big(\bigoplus_{[\lambda]} \Z/2\big)\oplus \big(\bigoplus_{\{[\mu],[\bar\mu]\}}
  \Z\big)$.
\end{proposition}

\begin{remark}
  \leavevmode\Nopagebreak
  \begin{enumerate}
  \item In the definition of $\Theta$, we have one summand for the two
    classes $[\mu]$ and~$[\bar\mu]$.  Here we have sign ambiguity
    since $\Theta_\mu=-\Theta_{\bar\mu}$, but this does not cause any
    problems in our conclusions.
  \item The homomorphisms $\Psi_\lambda$ and $\Psi$ extend to homomorphisms of
    $Q(H)^\times/\ANH$, which will also be denoted by
    $\Psi_\lambda$ and~$\Psi$.  Also, we denote by $\Theta$ and
    $\Theta_\mu$ the homomorphisms of $Q(H)^\times/\ANH$ induced by
    $\Theta$ and~$\Theta_\mu$.  Then the isomorphism in the last
    sentence of Theorem~\ref{theorem:algebraic-classification} can be
    written as $(\Psi,\Theta)$.
  \item Although $\Theta_\mu$ could be defined for self-dual $\mu$ as
    well, it is not interesting since the resulting $\Theta_\mu$ is
    always zero.
  \end{enumerate}
\end{remark}

\begin{proof}
  First we will observe that $\Psi_\lambda$ and $\Psi$ are
  well-defined surjective homomorphisms.  Since the factorization into
  irreducible factors is preserved by the $\Aut^*(H)$-action,
  $e_\lambda$ is invariant under~$\sim$ for any irreducible~$\lambda$.
  If $\lambda$ is self-dual, then $e_\lambda(u\bar u) =
  e_\lambda(u)+e_{\bar\lambda}(\bar u) = 2e_\lambda(u)$.  It follows
  that $\Psi_\lambda$ is a well-defined homomorphism.  The
  surjectivity of $\Psi_\lambda$ and $\Psi$ follows from the observation that
  $\Psi_{\lambda'}(\lambda\cdot \AH)$ equals 1 if $\lambda' \sim
  \lambda$, and 0 otherwise.

  To see that $\Theta_\mu$ and $\Theta$ are well-defined
  homomorphisms, observe that for $p\in Q(H)^{sym}$, $e_\mu(p) =
  e_{\bar\mu}(\bar{p}) = e_{\bar \mu}(p)$.  Their surjectivity now follows
  from the observation that  for non-self-dual irreducible $\mu$ and $\mu'$,
  $\Theta_{\mu'}(\mu\cdot Q(H)^{sym})$ equals $1$ if $\mu'\sim\mu$, $-1$
  if $\mu'\sim\bar\mu$, and 0 otherwise.

  For the injectivity of $\Psi$, suppose $f\in Q(H)^{sym}$ represents
  an element in the kernel of~$\Psi$.  We can rewrite the irreducible
  factorization of $f$ to obtain an expression
  $f=\lambda_1^{m_1}\cdots \lambda_r^{m_r}\cdot u$, where the
  $\lambda_i$ are mutually non-equivalent self-dual irreducible
  elements, $m_i\in \Z$, and $u\in \ANH$.  Note that $\lambda_i^2
  = \lambda_i\bar \lambda_i = 1$ in $Q(H)^{sym}/\ANH$ since
  $\lambda_i$ is self-dual.  Evaluating $\Psi_{\lambda_i}$, we have
  that $m_i$ is even for each~$i$.  From this it follows that $f=1$ in
  $Q(H)^{sym}/\ANH$.

  The injectivity of $\Theta$ is proved similarly: if $f\in
  Q(H)^\times$ represents an element in the kernel of $\Psi$, then
  from the irreducible factorization of $f$ we obtain an expression
  $f=\mu_1^{m_1}\bar\mu_1^{n_1}\cdots \mu_r^{m_r}\bar\mu_r^{n_r}\cdot
  u$, where $m_i, n_i\in \Z$, $u\in Q(H)^{sym}$, and the $\mu_i$ are
  non-self-dual irreducible elements such that $\mu_i\not\sim \mu_j
  \not\sim \bar\mu_i$ whenever $i\ne j$.  Evaluating $\Theta_{\mu_i}$,
  we have $m_i=n_i$ for each~$i$.  It follows that $f=1$ in
  $Q(H)^\times/Q(H)^{sym}$.
\end{proof}

Now, in order to understand the structure of $Q(H)^\times/\ANH$, the
only remaining part is to count the number of summands of $\Psi$ and
$\Theta$ in Proposition~\ref{prop:complete}.  To state the result, we
introduce the following definition: we say that $(g,n)$ is
\emph{small} if either $n\leq 1$ or $n=2$ and $g=0$.  Otherwise we say
that $(g,n)$ is \emph{large}.

\begin{theorem}
  \label{theorem:algebraic-classification}
Suppose $H$ is nontrivial, i.e., $\sgn$ is neither a sphere nor
disk. Then the following hold:
  \begin{enumerate}
  \item $\Psi$ is an isomorphism of $Q(H)^{sym}/\ANH$
    onto~$(\Z/2)^\infty$.
  \item If $(g,n)$ is small, then $Q(H)^\times/Q(H)^{sym}=0$.  If
    $(g,n)$ is large, then $\Theta$ is an isomorphism of
    $Q(H)^\times/Q(H)^{sym}$ onto~$\Z^\infty$.
  \end{enumerate}
Consequently,
  \[
  \frac{Q(H)^\times}{\ANH} \cong
  \begin{cases}
    (\Z/2)^\infty & \text{if $(g,n)$ is small,}\\
    (\Z/2)^\infty\oplus \Z^\infty & \text{if $(g,n)$ is large.}
  \end{cases}
  \]
\end{theorem}

\begin{proof}
  Recall that in Section~\ref{subsection:symmetry-asymmetry} we
  observed that there is no non-self-dual $\mu$ if $(g,n)$ is small.
  The first sentence of Theorem~\ref{theorem:algebraic-classification}
  (2) is an immediate consequence.  In the following
  subsections, we will realize, as the values of torsion invariants of
  homology cylinders, infinitely many self-dual classes $[\lambda]$
  when $H$ is nontrivial (see Theorem~\ref{theorem:realization}), and
  infinitely many non-self-dual classes $[\mu]$ when $(g,n)$ is large
  (see Theorem~\ref{theorem:realization-Z}).
\end{proof}

\subsection{Proofs of Theorems~\ref{mainthm} and \ref{mainthm2}}
\label{subsection:proofs}

In Sections~\ref{subsection:proofs} and
\ref{subsection:invariants-under-Aut*(H)}, we give proofs of
Theorems~\ref{mainthm} and \ref{mainthm2}. Along the way we also
conclude the proof of Theorem \ref{theorem:algebraic-classification}.

\begin{theorem}\label{theorem:realization}
  If $b_1(\sgn)>0$, then there exists a subgroup
  $\SS\subset \hgn$ isomorphic to $(\Z/2)^{\infty}$
  such that
  \[
  \SS\to \hgn \xrightarrow{\tau} Q(H)^\times/\ANH \xrightarrow{\Psi}
  \bigoplus_{[\lambda]} \Z/2
  \]
   is an injection whose
  image is the sum of a certain infinite set of $\Z/2$ summands of
  $\bigoplus_{[\lambda]} \Z/2$.
\end{theorem}

\begin{proof}
We pick non-trivial knots $K_i$ ($i=1,2\ldots$) which are negative
amphicheiral with irreducible Alexander polynomials such that the
multisets $C_i$ of nonzero coefficients of
$\Delta_i(t):=\Delta_{K_i}(t)$ are mutually distinct up to sign. For
example, one can use the family of knots described in
\cite[p.~60]{Ch07}: their Alexander polynomials are of the form
$a^2t^2-(2a^2+1)t+a^2$. It is well-known that the knots $K_i$ form a
$(\Z/2)$-basis of a subgroup of the knot concordance group
isomorphic to $(\Z/2)^{\infty}$.

We write $M=\S\times [0,1]$. Let $f\colon E_0\to \inte(M)$ be an
embedding as in Proposition~\ref{prop:embedc},  which is
homologically essential. Denote by $h$ the image of the generator of
$H_1(E_0)\cong \Z$ under the homomorphism $H_1(E_0) \to H$ induced
by the given embedding $f\colon E_0 \to M$.

We now write $M_i=M(f,K_i)$. By Proposition \ref{prop:embedc}, the
$M_i$ span a subgroup of $\hgn$ which is the homomorphic image of
$(\Z/2)^{\infty}$. Furthermore we have $\tau(M_i)=\Delta_i(h)$. Note
that each $\Delta_i(h)$ is irreducible and self-dual since
$\Delta_i(t)$ is irreducible and self-dual and since $h$ is easily
seen to be an indivisible element in $H$. It is not difficult to see
directly that the  multiset $C_i$ is an invariant of $\Delta_i(h)$
under the equivalence relation~$\sim$.  (For a more general method
from which this observation is derived as a special case, see
Section~\ref{subsection:invariants-under-Aut*(H)}.)  Therefore,
since the $C_i$ are all distinct, the equivalence classes of the
$\Delta_i(h)$ are mutually distinct.  From this we obtain
\[
  \Psi_{\Delta_i(h)}(\tau(M_j)) =
  \begin{cases}
    1 & \text{ if } i=j \\
    0 & \text{ otherwise.}
  \end{cases}
\]
Also, $\Psi_\lambda(\tau(M_i))=0$ if $\lambda \not\sim \Delta_i(h)$.
Therefore the composition $\SS \to \bigoplus_{[\lambda]} \Z/2$ in
the statement of this theorem is injective and has image
$\bigoplus_{[\Delta_i(h)]} \Z/2 \cong (\Z/2)^\infty$.
\end{proof}

We now obtain Theorem \ref{mainthm} as an immediate corollary.

\begin{proof}[Proof of Theorem \ref{mainthm}]
  By Theorem \ref{theorem:realization}, we have a subgroup $\SS$ of
  $\hgn$ and a homomorphism $\hgn \to (\Z/2)^\infty$ whose restriction
  to $\SS$ is an isomorphism.  It follows that the homomorphism
  splits, and $\SS\cong (\Z/2)^\infty$ descends to  a summand of the
  abelianization of~$\hgn$.
\end{proof}

\begin{remark}\label{remark:proof of mainthm}
  \leavevmode\Nopagebreak
  \begin{enumerate}
  \item Using the full power of
    Theorem~\ref{theorem:string-link-sum-formula}, i.e., by tying in
    string links with several components, we can realize many more
    nontrivial values of the homomorphism~$\Psi$.
  \item Following the arguments in the proof of Theorem~\ref{mainthm},
    one can easily show that if $b_1(\S)>0$ then there exists a
    commutative diagram
    \[
    \xymatrix{ (\Z/2)^{\infty}\ar[d]\ar[r]^{\id} & (\Z/2)^{\infty} \\
     \hgnsmooth \ar[r] & \hgn^{\top} \ar[u] }
    \]
    such that the left hand map  is injective and the right hand map is surjective.
  \end{enumerate}
\end{remark}

We also have the following realization result:

\begin{theorem}
  \label{theorem:realization-Z}
  If $(g,n)$ is large, then the image of
  \[
  \hgn\xrightarrow{\tau} Q(H)^\times/\ANH \xrightarrow{\Theta}
  \bigoplus_{\{[\mu],[\bar\mu]\}} \Z
  \]
  contains infinitely many summands of
  $\bigoplus_{\{[\mu],[\bar\mu]\}}\Z$.
\end{theorem}

The proof of Theorem \ref{theorem:realization-Z} requires a more
sophisticated method to detect non-equivalence of non-self-dual
irreducible factors.  This will occupy all of
Section~\ref{subsection:invariants-under-Aut*(H)}.  Assuming Theorem
\ref{theorem:realization-Z} we can now finally prove
Theorem~\ref{mainthm2}:

\begin{proof}[Proof of Theorem~\ref{mainthm2}]
Suppose that $n>1$.
    By Theorem~\ref{mainthm2-special}, we may assume that $(g,n)$ is large.
    Let $\hgn^{ab}$ be the abelianization
    of~$\hgn$. By Theorem~\ref{theorem:realization-Z}, we
  have a surjection $\hgn \to \Z^\infty$.  This induces a split
  surjection, say $g$, of the abelianization $\hgn^{ab}$ of $\hgn$
  onto $\Z^\infty$.  Also, by Theorem~\ref{theorem:realization}, we
  have a split surjection $f\colon \hgn^{ab} \to (\Z/2)^\infty$.
  Since the intersection of the images of the right inverses of
  $f$ and $g$ is automatically $\{0\}$ (e.g., compare the order), it
  follows that $(f,g)\colon \hgn^{ab} \to (\Z/2)^\infty \oplus
  \Z^\infty$ is a split surjection.
\end{proof}

\subsection{Detecting non-equivalent non-self-dual factors}
\label{subsection:invariants-under-Aut*(H)}

In order to prove Theorem \ref{theorem:realization-Z} we first will
introduce a simple and practical method for distinguishing  elements in
$\Z[H]$ up to the action of $\Aut^*(H)$.

As before, we denote $H=H_1(\Sigma)$ and $H_\partial=\Im\{H_1(\partial
\S) \to H\}$.  Denote $\widehat H=H/H_\partial=H_1(\widehat\Sigma)$
where $\widehat\Sigma$ is $\Sigma$ with boundary circles capped off,
and write $H=H_\partial \times \widehat H$ by choosing a splitting.
Fix a basis $\{x_1,\ldots,x_{n-1}\}$ of $H_\partial$, so that each
$u\in \Z[H_\partial]$ is viewed as a (Laurent) polynomial in the
variables~$x_i$.  For $u,v\in \Z[H_\partial]$, we write $u\approx v$
if $u=v\cdot m$ for some $m\in H_\partial$.  This is an equivalence
relation; denote the equivalence class of $u\in \Z[H_\partial]$
by~$[u]$.  Since $u\approx v$ if and only if the polynomial $u$ is
obtained from $v$ by shifting the exponents, it is very
straightforward to check whether $u\approx v$ or not for two given
polynomials $u$ and~$v$.  Given $p\in \zh$, write $p=\sum_{g\in
  \widehat H} u_g \cdot g$ where $u_g\in \Z[H_\partial]$, and define
\[
C(p) = \{[u_g] \mid g\in \widehat H\}.
\]
We view $C(p)$ as a multiset, i.e., repeated elements are allowed.
For $C(p)=\{[u_g]\}$, denote $-C(p) = \{[-u_g]\}$.

\begin{lemma}
  \label{lemma:aut-invariant}
  $C(p)$ is invariant up to sign under $\sim$ on $\Z[H]$, i.e., if
  $p\sim q$, $C(p)$ is equal to either $C(q)$ or $-C(q)$.
\end{lemma}
\begin{proof}
  Note that $\varphi\in \Aut^*(H)$ fixes $H_\partial$ and sends $g\in
  \widehat H$ to an element of the form $g'm_g$ for some $g'\in
  \widehat H$, $m_g\in H_\partial$.  In addition the association
  $g\mapsto g'$ is a bijection since $\varphi$ is an isomorphism.
  Therefore, for $p=\sum_{g\in \widehat H} u_g \cdot g \in \Z[H]$, the
  classes $[u_g]$ are permuted by the action of~$\varphi$.  It follows
  $C(p)=C(\varphi(p))$.  It is easily seen that $C(\pm p\cdot h)=\pm
  C(p)$ for $h\in H$.
\end{proof}

\begin{example}
  \label{example:total-coefficients}
  In the proof of Theorem~\ref{theorem:realization}, we have observed
  that for $p=\sum_{h\in H} c_h \cdot h\in \Z[H]$, the multiset of all
  nonzero coefficients $\{c_h\mid h\in H$ and $c_h\ne 0\}$ is an
  invariant, up to sign, of $p$ under $\sim$.  This can be viewed as a
  consequence of Lemma~\ref{lemma:aut-invariant}, since the multiset
  of nonzero coefficients of an element $u_g\in \Z[H_\partial]$ is
  invariant under~$\approx$.
\end{example}

We remark that if $C(p)=\{[u_g]\}$, then $C(\bar p) = \{[\bar u_g]\}$.
This combined with Lemma~\ref{lemma:aut-invariant} often allows us to
detect non-self-dual elements, as illustrated below.

\begin{example}
  \label{example:non-self-dual-element}
  Fix $g\ne e\in \widehat H$.  For a positive integer $a$, let
  \[
  p_a = 1+(g-1)x_i+g(g-1)x_i^2+\cdots+g^{a-1}(g-1)x_i^a.
  \]
  Then
  \begin{align*}
    C(p_a)&=\{[1-x_i], \ldots, [x_i^{a-1}-x_i^a], [x_i^a]\}
    = \{[1-x_i], \ldots, [1-x_i], [1]\},\\
    C(\bar p_a)&=\{[x_i-1], \ldots, [x_i-1], [1]\}.
  \end{align*}
  Looking at the element $[1]$ we see that $C(p_a)\ne -C(p_a)$.  Since $1-x_i
  \not\approx x_i-1$ we deduce that $C(p_a) \ne C(\bar p_a)$.  Therefore $p_a\not\sim
  \bar p_a$, i.e., $p_a$ is non-self-dual.  In particular, the torsion
  of the homology cylinder $M(a)$ in
  Section~\ref{subsection:exmaple-handle-decomp} is non-self-dual.
  Also, since $|C(p_a)| =|C(\bar p_a)|= a+1$, we have $p_a \not\sim
  p_b \not\sim \bar p_a$ whenever $a\ne b$.
\end{example}

We are now finally in a position to prove Theorem \ref{theorem:realization-Z}.

\begin{proof}[Proof of Theorem \ref{theorem:realization-Z}]
    First we consider the case
  $n=2$ and $g>0$.  For each positive integer $a$, consider the
  homology cylinder $M(a)$ constructed in
  Section~\ref{subsection:exmaple-handle-decomp}.  Let $p_a =
  \tau(M(a))\in \Z[H]$. As we observed in
  Example~\ref{example:non-self-dual-element}, $p_a$ is non-self-dual,
  and $p_a \not\sim p_b \not\sim \bar p_a$ whenever $a\ne b$.
  Applying the Eisenstein criterion to $p_a$, it can be seen that
  $p_a$ is irreducible.  Therefore
  \[
  \Theta_{p_a}(\tau(M(b))) =
  \begin{cases}
    1 &\text{if } a=b, \\
    0 &\text{otherwise}.
  \end{cases}
  \]
  From this it follows that the image of the subgroup generated by the
  classes of the $M(a)$ under the homomorphism
  \[
  \hgn\xrightarrow{\tau} Q(H)^\times/\ANH
  \xrightarrow{\Theta}\bigoplus_{\{[\mu],[\bar\mu]\}}\Z
  \]
  is equal to $\bigoplus_{\{[p_a],[\bar p_a]\}} \Z \cong \Z^\infty$.
  This proves the theorem for this case.

  For $n\ge 3$, the theorem is proved by a similar construction of
  homology cylinders, using the generator of $H$ associated to a third
  boundary component of $\Sigma$ in place of the generator~$y_1$.  (In
  this case the torsion $p_a$ lies in $\Z[H_\partial]$ and thus it is
  easier to see that $p_a \not\sim p_b \not\sim \bar p_a$ whenever
  $a\ne b$.)
\end{proof}

\section{Pretzel links}\label{section:pretzel}

In this section we study homology cylinders arising from pretzel
links.  The pretzel link \(P(2r,2s,2t)\) is a $3$--component link with
an obvious Seifert surface \(\Sigma\), as shown in
Figure~\ref{fig:Pretzel}.  We pick $\S$ as a model surface
for~$\S_{0,3}$.  We then write $M(r,s,t)=(S^3$ cut along
$\Sigma,i_+,i_-)$.

\begin{figure}[ht]
  \begin{center}
    \labellist
    \normalsize\hair 2pt
    \pinlabel {$a$} at 68 140
    \pinlabel {$b$} at 178 140
    \pinlabel {$r$} at 28 98
    \pinlabel {$s$} at 130 98
    \pinlabel {$t$} at 204 98
    \pinlabel {$\Sigma$} at 115 14
    \endlabellist
    \includegraphics[scale=0.8]{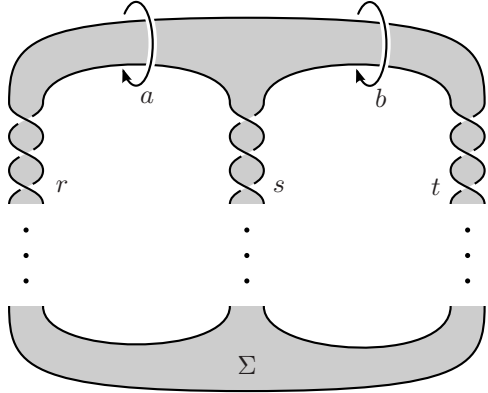}
  \end{center}
  \caption{The pretzel link $P(2r,2s,2t)$ with $r$, $s$, and $t$ full
    twists.}
  \label{fig:Pretzel}
\end{figure}

The main result of this section is the following.

\begin{proposition} \label{prop:pretzels}
  \leavevmode\Nopagebreak
  \begin{enumerate}
  \item $M(r,s,t)$ defines a homology cylinder if and only if
    $(r+s)(t+s)-s^2=\pm 1$.
  \item The set of homology cylinders $M(r,s,t)$ generates a
    $\Z^\infty$ subsummand of the abelianization of~$\HH_{0,3}$.
  \end{enumerate}
\end{proposition}

The proof of the proposition will require the remainder of this
section.  Given $r$, $s$, $t$, we write $M = S^3$ cut
along~$\Sigma$. Note that $M$ is a handlebody and that $\pi_1(M)$ is
the free group on the generators $a$ and $b$ as shown in
Figure~\ref{fig:Pretzel}.

Let \(\alpha\) be a loop on $\S$ which runs from  down the left-hand
strip, and back up via the middle strip. Similarly, let \(\beta\) be
a loop which runs down the right-hand strip and back up the middle,
so that  \(\pi_1(\Sigma)\) is generated by \(\alpha\) and \(\beta\).
We now denote by $\alpha^+$ and $\beta^+$ the corresponding curves
on $\S^+\subset M$. Then
\begin{equation*}
\alpha^+  = a^{r}(ab^{-1})^s \quad \quad
\beta^+  = b^{-t}(ab^{-1})^{s}.
\end{equation*}
We find that $H_1(\S)\to H_1(M)$ is an isomorphism if and only if
$(r+s)(t+s)-s^2=\pm 1$. This concludes the proof of Proposition
\ref{prop:pretzels} (1).

We will now calculate the torsions for homology cylinders with $(r+s)(t+s)-s^2=1$.
Note that this condition on $r,s,t$ is equivalent to $(r+s)(t+s)=s^2+1$.
In particular  one of $r,s$ or $t$
is necessarily negative.

Using the symmetries of Pretzel links we can without loss of
generality assume that $r,s>0$ and $t<0$.  Note that the condition
$s^2+1=(r+s)(t+s)$ implies that $|t|<s$ and $|t|<r$.

Recall that we can view $H_1(M(r,s,t))$ as the free abelian
multiplicative group with  basis $\{a,b\}$.  Using the isomorphism
$H_1(\S_{0,3})\to H_1(\S^+)\to H_1(M(r,s,t))$ we now identify
$H_1(\S_{0,3})$ with the multiplicative free abelian group generated
by $a$ and $b$.  In particular we will view $\tau(M(r,s,t))$ as an
element in $\zab$.

Evaluating the \(2\times2\) determinant
\[
\begin{bmatrix}
  \partial \alpha^+/\partial a & \partial \beta^+ / \partial a
  \\ \partial \alpha^+ / \partial b &\partial \beta^+/ \partial b
\end{bmatrix}
\]
we obtain from \cite[Proposition~4.2]{FJR09} that $\tau(M(r,s,t))$
equals
\begin{equation}
  \ba{l@{}r@{}l} &&(1+a+\dots+a^{r-1})(1+b+\dots+b^{|t|-1})\\
  +\,&a^r& (1+b+\dots+b^{|t|-1})(1+ab^{-1}+\dots+(ab^{-1})^{s-1})\\
  -\,&ab^{|t|-1}& (1+a+\dots+a^{r-1})(1+ab^{-1}+\dots+(ab^{-1})^{s-1}).\ea
  \tag{$*$}\label{comptau}
\end{equation}
This polynomial is always asymmetric.  In fact the support of the
polynomial $\tau(M(r,s,t))$ is given in Figure \ref{fig:support}.
Here a point $(i,j)$ corresponds to $a^ib^j$.  Furthermore a triangle
($\blacktriangle$) indicates a coefficient $-1$ and a big five-pointed star
($\bigstar$) indicates a coefficient $+1$.

\begin{figure}[ht]
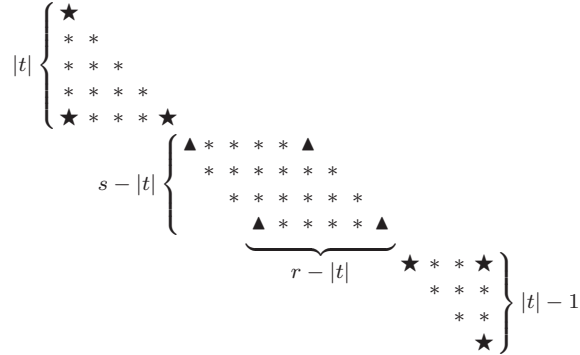

  \begin{center}
    {\footnotesize$
    \arraycolsep=.3em
    \def\bs{\hbox to 0mm{\hss$\bigstar$\hss}}
    \def\tr{\hbox to 0mm{\hss$\blacktriangle$\hss}}
    \begin{array}{r@{}r@{}c}
      |t| \left\{
        \begin{array}{ccccc}
          \bs \\
          * & * \\
          * & * & * & \\
          * & * & * & * & \hphantom{*}\\
          \bs & * & * & * & \bs
        \end{array}
      \right.
      \\
      s-|t| \left\{
        \vphantom{\begin{array}{c} * \\ * \\ * \\ *\end{array}}
      \right. &
      \begin{array}{ccccccccc}
        \tr & * & * & * & * & \tr \\
        & * & * & * & * & * & * \\
        &   & * & * & * & * & * & * & \hphantom{*} \\
        &   &   & \tr & * & * & * & * & \tr
      \end{array}
      \\
      & \begin{array}{@{}r@{}}
        \vbox to 0mm{\vss\hbox{$\underbrace{
              \hphantom{\begin{array}{cccccc}*&*&*&*&*&*\end{array}}}
            _{\displaystyle r-|t|}$}%
        \vss\vss}
        \\ \\ \\ \\ \\
        \end{array}
      & \left.
        \begin{array}{cccc}
          \bs & * & * & \bs \\
          \hphantom{*} & * & * & * \\
          &   & * & * \\
          &   &   & \bs
        \end{array}
      \right\} |t|-1
    \end{array}
    $}
    \caption{Support of $\tau(M(r,s,t))$.}
    \label{fig:support}
  \end{center}
\end{figure}

\begin{proposition}\label{prop:irrfactor}
  There exist infinitely many positive integers $r_i,s_i$, negative
  integers $t_i$, and irreducible polynomials $p_i\in \zab$ with the
  following properties:
  \begin{enumerate}
  \item $s_i^2+1=(r_i+s_i)(t_i+s_i)$,
  \item $p_i\not\doteq \ol{p_i}$, and
  \item $p_i$ divides $\tau(M(r_j,s_j,t_j))$ if and only if $i=j$.
  \end{enumerate}
\end{proposition}

\begin{proof}
  Pick distinct odd primes $x_1,x_2,\dots $.  Furthermore pick
  $r_i>x_i$ and $s_i>x_i$ such that $(r_i-x_i)(s_i-x_i)=1+x_i^2$.  We
  then have
  \begin{align*}
    (r_i+s_i)(s_i-x_i)&=((r_i-x_i)+(s_i+x_i))(s_i-x_i)\\
    &=(s_i+x_i)(s_i-x_i)+(r_i-x_i)(s_i-x_i)\\
    &=s_i^2-x_i^2+1+x_i^2\\
    &=s_i^2+1.
  \end{align*}
  We now set $t_i=-x_i$, it follows that $s_i^2+1=(r_i+s_i)(t_i+s_i)$.

  In order to show the existence of $p_i$ with the required properties
  we have to introduce various definitions.  Given $p\in \zbthena$ we
  now denote by $l(p)\in \zb$ the coefficient of the lowest degree and
  by $h(p)\in \zb$ the coefficient of the highest degree.  Note that
  for $p,q\in \zbthena$ we have $l(p\cdot q)=l(p)\cdot l(q)$ and
  $h(p\cdot q)=h(p)\cdot h(q)$.

  We now write $\tau_i=\tau(M(r_i,s_i,t_i))$ and we view $\tau_i$ as
  an element in $\zbthena$.  By Equation (\ref{comptau}) we have
  \[
  l(\tau_i)=1+b+b^2+\dots+b^{x_i-1}.
  \]
  Since $x_i$ is prime it follows that $l(\tau_i)\in \zb$ is
  irreducible.  In particular for any $i$ there exists an irreducible
  factor $p_i\in \zbthena$ of $\tau_i$ with
  $l(p_i)=1+b+b^2+\dots+b^{x_i-1}$.

  Note that $h(p_i)$ divides $h(\tau_i)=1+b+\dots+b^{x_i-2}$. In
  particular $\deg(h(p_i))=x_i-2$ and $\deg(l(p_i))=x_i-1$.  It
  follows easily that $p_i\not\doteq \ol{p_i}$.

  Given $i,j$ we have $l(p_i)=1+b+\dots+b^{x_i-1}$ and
  $l(\tau_j)=1+b+\dots+b^{x_j-1}$.  In particular if $i\ne j$, then
  $l(p_i)$ and $l(\tau_j)$ are distinct irreducible polynomials.  It
  follows that $p_i$ does not divide $\tau_j$ if $i\ne j$.
\end{proof}

Write $H=H_1(\Sigma_{0,3})$.  Note that $\Aut^*(H)=\{\id\}$. In
particular the polynomials $p_i\in \zab=\zh$ satisfy $p_i\not\sim
\ol{p_i}$.  It now follows immediately that the homology cylinders
$M(r,s,t)$ span a subgroup of $\mathcal{H}_{0,3}$ which surjects onto
$\Z^{\infty}$ under the map $\Theta\circ \tau$.  This concludes the
proof of Proposition \ref{prop:pretzels}.

\section{The Torelli subgroup}
\label{section:torelli}

In this section, we consider a subgroup of $\hgn$ which generalizes
the Torelli group of the mapping class group. We prove analogues of
Theorems~\ref{mainthm} and \ref{mainthm2}, but for a larger set of
surfaces.

Let $g,n\geq 0$ and write $H=H_1(\sgn)$. Recall that each
$\varphi\in\mgn$ induces an action $\varphi_*$ on $H$ and the map
$\mgn\to \aut^*(H)$ sending $\varphi$ to $\varphi_*$ is an
epimorphism. The Torelli group $\mathcal{I}_{g,n}$ is defined to be
the kernel of this map, i.e. $\ign$ is the subgroup of $\mgn$ given by
all elements which act as the identity on $H$. We refer to
\cite{Jo83b} and \cite{FM09} for details on the Torelli group. The
following theorem summarizes some of the key properties of the Torelli
group whose proofs can be found in \cite[Theorem~7.10]{FM09},
\cite{Jo83a}, \cite{Me92}.  (See also~\cite{MM86}), and \cite{Jo85}.)

\begin{theorem}
  \leavevmode\Nopagebreak
  \begin{enumerate}
  \item The group $\II_{g,n}$ is torsion-free,
  \item the group $\II_{g,n}$ is finitely generated for $g\geq 3$ and
    $n=0,1$,
  \item the group $\II_{2,0}$ is a free group on infinitely many
    generators,
  \item if $g\geq 3$, then the abelianization of $\II_{g,1}$ is
    isomorphic to $\Z^a\oplus (\Z/2)^b$ for some $a,b\in \N$.
  \end{enumerate}
\end{theorem}

It is an open question though whether the Torelli group $\II_{g,1}$ is
finitely presented for $g\geq 3$ (e.g., see \cite[Section~7.3]{FM09}).

Now recall that the action of a homology cylinder on $H=H_1(\sgn)$
also gives rise to an epimorphism $\varphi\colon \hgn \to
\aut^*(H)$. We now define the \emph{Torelli group $\ihgn$ of homology
  cylinders over $\sgn$} to be the kernel of~$\varphi$.  By
Proposition \ref{prop:mgnembed} we can view the Torelli group $\ign$
as a subgroup of $\ihgn$.

It is an immediate consequence of Proposition \ref{prop:deltasum} and
Theorem \ref{thm:homcob} that the torsion function gives rise to a
homomorphism
\[
\tau\colon \ihgn \to Q(H)^\times/\NH.
\]
Note that in this case we do not need to divide $Q(H)^\times$ out by
$\AH=A(H)$ since $\varphi(M)$ is the identity for any $M\in \ihgn$.  We
can now prove the analogues of Theorems~\ref{mainthm} and
\ref{mainthm2}.

\begin{theorem}\label{thm:cyltorelli}
  \leavevmode\Nopagebreak
  \begin{enumerate}
  \item If $b_1(\sgn)>0$, then
    there exists an epimorphism
    \[
    \ihgn \to (\Z/2)^\infty
    \]
    which splits. In particular, the abelianization of $\ihgn$
    contains a direct summand isomorphic to~$(\Z/2)^\infty$.
  \item If $g>1$ or $n>1$, then
    there exists an epimorphism
    \[
    \ihgn \to \Z^\infty.
    \]
    Furthermore, the abelianization of $\ihgn$ contains a direct summand
    isomorphic to $(\Z/2)^\infty\bigoplus\Z^\infty$.
  \end{enumerate}
\end{theorem}

\begin{remark}
\begin{enumerate}
\item
Note  that provided the genus is at least two
Theorem~\ref{thm:cyltorelli} (2) holds also for closed surfaces and
for surfaces with one boundary component. This is in contrast to the
situation in Theorem \ref{mainthm2}.
\item
Morita (\cite[Corollary~5.2]{Mo08}) used `trace maps' to show that
the abelianization of $\mathcal{IH}_{g,1}, g\geq 1$ has infinite
rank. Theorem \ref{thm:cyltorelli} (2) can therefore be seen as an
extension of Morita's theorem.
\end{enumerate}
\end{remark}

\begin{proof}
Part (1) follows immediately from the proof of Theorem \ref{mainthm}
since the examples $M_i, i\in \N$, provided in the proof are easily
seen to lie in $\ihgn$.

We  will now need the following claim:

\begin{claim}
Given any $M\in \hgn$ there exists $M'\in \ihgn$ with
$\tau(M')=\tau(M)\in \zh$.
\end{claim}

Indeed, let $M=(M,i_+,i_-)$ be a homology cylinder over $\sgn$. As we
saw in Section \ref{section:actionh1} we have $\varphi(M)\in
\aut^*(H)$, and there exists $\psi\in \mgn$ such that the induced
action on $H_1(\sgn)$ is given by $\varphi(M)$. Now the homology
cylinder $M'=(M,i_+,i_-\circ \psi^{-1})$ acts as the identity on $H$,
i.e. $M'$ defines an element in $\ihgn$. On the other hand it is clear
that $\tau(M')=\tau(M)$. This concludes the proof of the claim.

We now turn to the proof of Part (2).
First suppose that $n>1$ and $(g,n) \ne (0,2)$. Part (2) for this
case follows from the proof of Theorem~\ref{mainthm2} since the
examples $M(a), a\in \N$, in the proof of
Theorem~\ref{theorem:realization-Z} can be realized as elements in
$\ihgn$ by the above claim. When $(g,n)=(0,2)$, following the
arguments in Section~\ref{section:examples} one can easily see that
$\mathcal{IH}_{0,2}$ is isomorphic to $\CC_{\Z}$ and the desired
result follows again from Levine's work \cite{Le69a, Le69b}.

Finally suppose that $g>1$. In this case, for $a\in \N$, we consider the
homology cylinder $M(a)$ constructed using
Figure~\ref{figure:example-attaching-curve} modified in the
following way: in Figure~\ref{figure:example-attaching-curve} we
have two boundary components, which are connected by $x^*$, and one
hole to which a generator $y_1^*$ is associated. Now remove the
second boundary component and replace the first boundary component
by a new hole, and denote by $x^*$ a closed curve which `connects' the two holes.
It follows from the discussion of Section
\ref{subsection:exmaple-handle-decomp} and from the above claim that
there exists a homology cylinder $M(a)$ which lies in the Torelli
group $\ihgn$ and such that
\[ \tau(M(a))=p_a:= 1+(y-1)x+y(y-1)x^2+\cdots+y^{a-1}(y-1)x^a.\]

Recall that the polynomials $p_a\in \zh$ are irreducible.  Also it is
evident that the $p_a$ are non-symmetric (i.e., $p_a\not\doteq
\bar{p}_a$) and $p_a \not\doteq p_b$ for $a\ne b$. (Note that $p_a\sim
p_{\ol{a}}$, and therefore Theorem~\ref{mainthm2} cannot be
generalized for this case using our previous method.) To detect
$\tau(M(a))\in Q(H)^\times/\NH$, for each irreducible $\mu\in \zh$ we
use the homomorphism $\Theta_\mu\colon Q(H)^\times/\NH \to \Z$ defined
as in Section~\ref{section:algeraic structure} with the following
modification: $e_\mu(p)=$ the exponent of $\mu$ in $p$. Then
$\Theta_{p_a}(M(b))=1$ if $a=b$ and 0 otherwise. Now following the
lines of the proof of Theorem~\ref{mainthm2}, Part (2) follows.
\end{proof}

Note that Theorem~\ref{thm:cyltorelli} (1) also implies that if
$b_1(\sgn)> 0$, then $\ihgn$ is neither finitely generated nor
finitely related, and it is not torsion-free. Also for the structure
of the group $Q(H)^\times/N(H)$, redefining $\Psi_\l$ as we did for
$\Theta_\mu$ in the proof of Theorem~\ref{thm:cyltorelli}, one can
easily obtain the following analogue of
Theorem~\ref{theorem:algebraic-classification}:

\begin{theorem}
  \label{theorem:algebraic-classification2}
  Suppose $H$ is nontrivial, i.e., $\sgn$ is neither a sphere nor disk. Then
  \[
  Q(H)^\times/N(H)
  \cong
  \begin{cases}
    (\Z/2)^\infty & \text{if $g,n\le 1$ or $(g,n)=(0,2)$,}\\
    (\Z/2)^\infty\oplus \Z^\infty & \text{otherwise.}
  \end{cases}
  \]
\end{theorem}

\section{Questions}

We conclude with a short list of questions:

\begin{enumerate}
\item Study the structure of the kernel of
  \[
  \tau\colon \hgn \to
  Q(H)^\times/\ANH \cong (\Z/2)^\infty \oplus \Z^\infty.
  \]
\item Does the abelianization of the group $\HH_{g,n}$ have  infinite rank for $g>0$
  and $n= 0,1$?
\item Characterize the image of the homomorphisms $\tau\colon \hgn \to
  Q(H)^\times/\ANH$, $\Psi\colon \hgn \to (\Z/2)^\infty$ and
  $\Theta\colon \hgn \to \Z^\infty$.
\item Can we realize any element in the image of $\Psi$ by using
  Theorem~\ref{theorem:string-link-sum-formula}?
\item Does the abelianization of $\hgn$ contain $4$--torsion? Does it
  contain any other torsion?
\item   Let $g>0$ and $n\geq 0$. Does there exist a homomorphism
  $\FF\colon\HH_{g,n}^{\smooth}\to A$ onto an abelian group of
  infinite rank such that  the restriction of $\FF$ to the kernel
  of the projection map $\HH_{g,n}^{\smooth}\to \HH_{g,n}^{\top}$ is
  also surjective?
\item Does there exist a monoid homomorphism from $\cgn$ onto a
  non-abelian monoid which vanishes on $\mgn$?
\item Is the group $\hgn/\ll \mgn\rr$ non-abelian?
\end{enumerate}

\end{document}